\documentclass[leqno,english]{amsart}
\usepackage{amsmath,amssymb,amsthm,enumerate,esint,mathtools,xparse}
\usepackage[svgnames]{xcolor}
  \usepackage{hyperref}
 \usepackage{empheq}

  \usepackage{bbm}

 \usepackage{mathrsfs} 
\usepackage{mathabx}

\numberwithin{equation}{section}
\newtheorem{definition}{Definition}[section]
\newtheorem{theorem}{Theorem}[section]

\newtheorem{lemma}[theorem]{Lemma}
\newtheorem{proposition}[theorem]{Proposition}

\newtheorem{remark}[theorem]{Remark}

\newcommand{\bke}[1]{\left ( #1 \right )}
\newcommand{\bkt}[1]{\left [ #1 \right ]}

\newcommand{\norm}[1]{\left \| #1 \right \|}
\newcommand{\bka}[1]{{\langle #1 \rangle}}
\newcommand{\abs}[1]{\left | #1 \right |}

\newcommand\al{\alpha}

\newcommand\ga{\gamma}
\newcommand\de{\delta}
\newcommand\ep{\epsilon}

\renewcommand\th{\theta}
\newcommand\ka{\kappa}
\newcommand\la{\lambda}
\newcommand\si{\sigma}

\newcommand\om{\omega}

\newcommand\De{\Delta}

\newcommand\Om{\Omega}
\newcommand\Rg{\mathscr{R}_g}

\newcommand{\R}{\mathbb{R}}

\renewcommand{\div}{\mathop{\rm div}\nolimits}

\newcommand{\pd}{\partial}
\newcommand{\nb}{\nabla}
\newcommand{\td}{\tilde}

\renewcommand{\bar}[1]{\overline{#1}}
\newcommand{\lec}{{\ \lesssim \ }}

\newcommand{\EQ}[1]{\begin{equation}\begin{split} #1 \end{split}\end{equation}}
\newcommand{\EQN}[1]{\begin{equation*}\begin{split} #1 \end{split}\end{equation*}}
\newcommand{\EN}[1]{\begin{enumerate} #1 \end{enumerate}}

\DeclarePairedDelimiter{\oldnormaux}{\bracevert}{\bracevert}

\NewDocumentCommand{\oldnorm}{som}{%
  \IfBooleanTF{#1}
    {\oldnormaux*{#3}}
    {\IfNoValueTF{#2}
       {\oldnormaux*{\vphantom{dq}#3}}
       {\oldnormaux[#2]{#3}}%
    }%
}

\begin{document}
\title
{
Asymmetric deformations of a perturbed spherical bubble in an incompressible fluid
}

\author[C.-C. Lai]{Chen-Chih Lai}
\address{\noindent
Department of Mathematics,
Columbia University , New York, NY, 10027, USA}
\email{cl4205@columbia.edu / cclai.math@gmail.com}

\author[M. I. Weinstein]{Michael I. Weinstein}
\address{\noindent
Department of Applied Physics and Applied Mathematics and Department of Mathematics,
Columbia University , New York, NY, 10027, USA}
\email{miw2103@columbia.edu}

\begin{abstract} 
We study the dynamics of a gas bubble in a fluid with surface tension, starting near a spherical equilibrium. 
While there are many studies and applications of radial bubble dynamics, the theory of general deformations from a spherical equilibrium is less developed.  
We aim to understand how asymmetrically perturbed equilibrium bubbles evolve toward spherical equilibrium due to thermal or viscous dissipation in an incompressible liquid.

We focus on the 
 {\it isobaric approximation} \cite{Prosperetti-JFM1991}, under which the gas pressure within the bubble is spatially uniform and obeys the ideal gas law.
The liquid outside the bubble is incompressible,  irrotational, and has surface tension. 
We prove that any equilibrium gas bubble must be spherical by showing that the bubble boundary is a closed surface of constant mean curvature.

We then study the initial value problem (IVP) for the coupled PDEs, constitutive laws and interface conditions of the isobaric approximation for general (asymmetric) small initial perturbations of the spherical bubble in the linearized approximation. Our first result, considering thermal damping without viscosity, proves that the linearized IVP is globally well-posed. The monopole (radial) component of the perturbation decays exponentially over time, while the multipole (non-radial) components undergo undamped oscillations. This indicates a limitation of the isobaric model for non-spherical dynamics.  Our second result, incorporating viscous dissipation, shows that the IVP is linear and nonlinearly ill-posed due to an incompatibility of normal stress boundary conditions \underline{for non-spherical solutions}  and the irrotationality assumption. 
Our study concludes that to accurately capture the dynamics of general deformations of a gas bubble, the model must account for either vorticity generated at the bubble-fluid boundary, spatial non-uniformities in the gas pressure, or both.
 \end{abstract}
\maketitle

\tableofcontents

\section{Introduction}

The dynamics of a gas bubble in a liquid, when perturbed from a spherical equilibrium,   is a question of fundamental interest in fluid dynamics and its applications  \cite{Leighton-IJMPB2004, Leighton-book2012, Brennen-book2014}.
While there are many studies and applications of purely radial  bubble dynamics, the theory of general deformations from a spherical equilibrium  is much less developed. Non-spherical (shape) deformations are known to play an important role in  physical phenomena \cite{BLD-PRL1995, BKTW-RSL1999, Leighton-book2012, Brennen-book2014}.

Physical intuition suggests that, due to  dissipative mechanisms, as time advances the asymmetrically perturbed  equilibrium bubble will evolve toward a spherical equilibrium. We study this question in the setting of an incompressible liquid, in which the dissipation mechanisms are thermal diffusion or viscosity.  In particular, we study general small deformations of a nearly spherical gas bubble, governed by an ideal gas law, 
  in an incompressible fluid  under the influence of thermal and viscous dissipation mechanisms.  Other dissipation mechanisms, for example  radiation damping via acoustic waves in the case where the liquid is compressible (see, for example, \cite{KE-1972, KK-JAP1956, KM-JASA1980, LP-JFM1987, PL-JFM1986, WB-JFM2010, WB-JFM2011,SW-SIMA2011,CTW-SIMA2013}), are not considered here.
  
An approximation, which is valid in many physical situations, originating in the work of A. Prosperetti \cite{Prosperetti-JFM1991},  is the {\it isobaric approximation}, for which the gas pressure within the bubble evolves with time but is spatially uniform. Further, the fluid is incompressible and has surface tension; see also  \cite{bv-SIMA2000}. In \cite{LW-vbas2022, LW-vbaslinear2023} we studied, within this approximation, the global in time  nonlinear dynamics of a spherical gas bubble subject to radially symmetric perturbations; the manifold of spherical bubble equilibria, which is parametrized by the bubble mass, is nonlinearly exponentially stable.

 In this paper we initiate the in-depth analytical study of the isobaric approximation in the setting of general non-spherically symmetric perturbations. 
We focus on the case where the liquid is irrotational; the vorticity in the liquid is equal to zero. We prove that any equilibrium gas bubble must be spherical by showing that the bubble boundary is a closed surface of constant mean curvature.

 
We then  turn to the dynamics of  infinitesimal perturbations of the spherical equilibrium, which are governed by the linearized evolution equation.
Our first result concerns the case where there is thermal damping but no viscous dissipation. We prove that the initial value problem (IVP) for the linearized evolution is well-posed globally in time. Further, we find that the monopole  (radial) component of perturbation damps toward zero exponentially fast as time tends to infinity. In contrast,  the multipole components of the perturbation merely undergo undamped  oscillations with time, as in the case of a polytropic gas bubble in an inviscid and incompressible fluid with no damping mechanisms  \cite{Rayleigh-London1917,Lamb-book}.  We demonstrate that this is due to the spatial uniformity of the gas pressure. Hence, our results suggest that the effect of thermal damping on shape modes needs to be captured either in the framework where the vorticity is non-zero or in corrections to the uniform gas pressure (isobaric) model. 

Our second result concerns the linearized time-evolution where,  in addition to thermal damping, we allow viscous dissipation. We prove that the IVP in this case is (linearly and nonlinearly)  ill-posed.  This is consistent with the physical expectation that general deformations near the boundary of a \underline{viscous}  fluid  will generate vorticity 
\cite{Joseph-JFM2003, Joseph-IJMF2006}.
Mathematically this is manifested by the fact that a non-zero viscosity component of the momentum stress tensor implies normal stress boundary conditions at the fluid-bubble interface which are  incompatible with the assumption of irrotationality.  We note that a property of the purely radial dynamics is that the  radial / monopole component of the perturbation undergoes both thermal and viscous damping (see Proposition \ref{prop-l=0} and \cite{LW-vbaslinear2023}). However, a smooth and radial velocity field is necessarily irrotational, so there is no contradication.

\subsection{Structure of the paper and overview of results}
{\ }\smallskip


In Section 2 we give a brief introduction of the mathematical formulation for the uniform gas pressure (isobaric approximation) model.

In Section 3 we present this approximate model for the case of an irrotational fluid. In the irrotational framework, we prove that surface tension is sufficient to force any equilibrium bubble into a spherical shape (Theorem \eqref{thm-irro-equilib}). In contrast if allows for the fluid flow to be rotational, there are examples 
of non-spherically symmetric equilibrium bubbles; see Remark \ref{rem:rotational-examples} and the forthcoming article \cite{LW-aeb}.
 
In Section 4 we formulate the linearized initial value problem for the isobaric approximation and state our main results:

(I) \emph{Theorem \ref{thm-wellposed}:\ Linear well-posedness and long-term behavior of the solution in the inviscid case.}
When viscosity is set to zero, the IVP for the  linearized evolution is well-posed.
Further, while the monopole / radial mode decays exponentially in time, the multipole / shape modes undergo undamped oscillations.
Since the shape modes do not decay, this points to a limitation of the isobaric model for describing asymmetric dynamics.  

(II) \emph{Theorem \ref{thm-linear-illposed}:\ Ill-posedness of the IVP in the viscous and irrotational setting.}
The viscous and thermally dampled linearized model is ill-posed for general (nonspherically symmetric) initial data.
The source of ill-posedness is an incompatibility of the 
 of the normal stress boundary conditions, \underline{for non-spherical solutions},  and the irrotationality assumption; any regular solution is constrained to be spherically symmetric. 
 
Our study therefore shows that to capture the correct dynamics of general deformations of a gas bubble, one requires a model which accounts for either 
vorticity generated at the bubble-fluid boundary or spatial non-uniformities in the gas pressure or both.


In Section 5 we present a detailed proof of Theorem \ref{thm-wellposed}  via spherical harmonic and multipole expansions.  

In Section 6 we prove on linear ill-posedness. 

In Section 7, we build  on Theorem \ref{thm-linear-illposed} to prove  nonlinear ill-posedness of the viscously damped irrotational dynamics.

\subsection{Some future directions and open problems}
{\ }\smallskip

1. \emph{Nonlinear well-posedness of the approximate model of Prosperetti in the inviscid case.}
In the present paper, the approximate model is proved to be nonlinearly ill-posed in the viscous case and linearly well-posed in the inviscid case.
A natural direction for future research is to explore the nonlinear well-posedness of the model in the inviscid and thermally damped case.

2. \emph{Develop a  model and analysis of  thermally damping  shape modes.}
In this paper, we prove that for isobaric (uniform gas pressure) approximation,   the shape / multipole modes of bubble do not decay. 
In ongoing research \cite{LW-lsf} we consider a physically more accurate model that allows for  spatial variations for bubble  gas pressure.
Naturally, we propose replacing the uniform gas pressure assumption with the momentum equation:
\[
\rho_g\pd_t{\bf v}_g + \rho_g{\bf v}_g\cdot\nb{\bf v}_g = -\nb p_g,
\]
where $\rho_g$, ${\bf v}_g$, and $p_g$ are the gas density, velocity, and pressure, respectively.

3. \emph{Dynamic stability of rotational bubble (viscous effects in shape modes).}
Theorem \ref{thm-irro-illposed} shows that the dynamics of the irrotational and invsicid fluid / bubble system is ill-posed. 
Therefore, one needs to study the general (rotational) model in the viscous case.
An approach  is to employ the Helmholtz decomposition to decompose the flows into a rotational and an irrotational parts;  see, for example, \cite{Joseph-PNAS2006}.

4. \emph{Nonlinear dynamics of a bubble in a compressible fluid; acoustic radiation damping mechanism.} Consider the situation where the surrounding liquid is compressible. 
For the problem where there is neither thermal nor viscous damping mechanisms, it has been shown in \cite{SW-SIMA2011} that the spherically symmetric equilibrium bubbles are linearly asymptotically stable via the emission,  spreading and decay of acoustic radiation. It is natural to seek a nonlinear asymptotic stability theory in this setting.

%

\bigskip\noindent{\bf Acknowledgements.}
The authors thank Andrea Prosperetti, Michael Miksis, Juan J. L. Vel\'azquez, and Robert Pego for very stimulating discussions.
CL and MIW are supported in part by the Simons Foundation Math + X Investigator Award \#376319 (MIW). 
CL is also supported by AMS-Simons Travel Grant.
MIW is also supported in part by National Science Foundation Grant DMS-1908657 and DMS-1937254.

\section{Asymmetric dynamics of the model with uniform gas pressure}\label{appx-old-model}

We consider the spatial uniform (isobaric) approximate model of Prosperetti \cite{Prosperetti-JFM1991}; see also \cite[(3.1)--(3.3)]{LW-vbas2022}). This model describes the evolution of  deforming gas bubble which occupies a region 
simply connected subset of $\R^3$, $\Omega(t)$,  the  evolving gas within the bubble and the surrounding incompressible fluid. The boundary of the bubble, $\partial\Omega(t)$, is parametrized by a function $\boldsymbol{\om}:\xi\in\mathbb S^2\mapsto \boldsymbol{\om}(\xi,t)\in\mathbb R^3$.  The liquid is described by ${\bf v}_l$, the liquid velocity, $p_l$ and the liquid pressure. The gas is described by  $\rho_g$, the gas density, ${\bf v}_g$, the gas velocity, $p_g$, the gas pressure, $T_g$, the gas temperature, and $s$, the specific entropy. These are related by the following system of PDEs, constituitive laws and fluid-gas interface conditions
 \begin{subequations}\label{eq1.1simplified}
\begin{empheq}[right=\empheqrbrace\text{in $\R^3\setminus \Om(t)$, $t>0$,}]{align}
\pd_t {\bf v}_l + {\bf v}_l\cdot\nb{\bf v}_l=&\, \frac{\mu_l}{\rho_l}\De{\bf v}_l - \dfrac1{\rho_l}\, \nb p_l, \label{eq1.1simplified-a}\\
\div {\bf v}_l =&\, 0, \label{eq1.1simplified-b}
\end{empheq}
\end{subequations}

 \begin{subequations}\label{eq1.2simplified}
\begin{empheq}[right=\empheqrbrace\text{in $\Om(t)$, $t>0$,}]{align}
\pd_t \rho_g + \div(\rho_g{\bf v}_g) =&\, 0,\label{eq1.2simplified-a}\\
p_g =&\, p_g(t), \label{eq1.2simplified-b}\\
\rho_g T_g \bke{\pd_t s + {\bf v}_g\cdot\nb s } =&\, \div(\ka_g\nb T_g), \label{eq1.2simplified-c}\\
p_g =&\, \Rg T_g \rho_g , \label{eq1.2simplified-d}\\
s =&\, c_v \log\bke{\dfrac{p_g}{\rho_g^\ga} } \label{eq1.2simplified-e},
\end{empheq}
\end{subequations}
and
 \begin{subequations}\label{eq1.3simplified}
\begin{empheq}[right=\empheqrbrace\text{on $\pd\Om(t)$, $t>0$,}]{align}
{\bf v}_l(\boldsymbol{\om},t)\cdot\hat{\bf n} = {\bf v}_g(\boldsymbol{\om},t)\cdot\hat{\bf n} = \pd_t{\boldsymbol{\om}}\cdot\hat{\bf n}, \label{eq1.3simplified-a}\\
p_g \hat{\bf n} - p_l \hat{\bf n} + 2\mu_l\hat{\bf n}\cdot\mathbb{D}({\bf v}_l) = \si \hat{\bf n} (\nb_S\cdot \hat {\bf n}), \label{eq1.3simplified-b}\\
T_g = T_\infty, \label{eq1.3simplified-c}
\end{empheq}
\end{subequations}
Here, $\hat{\bf n}$ is the unit outer normal on $\pd\Om(t)$, $\mathbb{D}({\bf v}_l) = (\nb{\bf v}_l+\nb{\bf v}_l^\top)/2$ is the deformation tensor of the liquid, and $\nb_S\cdot$ denotes the surface divergence so that $\nb_S\cdot\hat{\bf n}$ is twice the mean curvature on the surface.

The system depends on the following physical parameters:~the density of the liquid $\rho_l>0$, the dynamic viscosity of the liquid $\mu_l\ge0$, the thermal conductivity of the gas $\ka_g\ge0$, the surface tension $\si$, far-field liquid temperature $T_\infty>0$, the specific gas constant $\Rg>0$, the heat capacity of the gas at constant volume $c_v>0$, and the adiabatic constant $\ga>1$. Here,  $\Rg$, $c_v$, and $\ga$ are related by $\ga = 1 + \frac{\Rg}{c_v}$.

The model \eqref{eq1.1simplified}-\eqref{eq1.3simplified} has been studied extensively in Prosperetti \cite{Prosperetti-JFM1991}, Biro--Vel\'azquez \cite{bv-SIMA2000}, and the authors in \cite{LW-vbas2022} in the setting of spherically symmetric solutions.

In this paper, we discuss the nonspherical dynamics in model \eqref{eq1.1simplified}--\eqref{eq1.3simplified}.
Note that one can eliminate $T_g$ and $s$ use $p=p_g(t)$ to simplify 
the equations of the gas \eqref{eq1.2simplified} to (see \cite[(B.6)]{LW-vbas2022})
Using the relation \eqref{eq1.2simplified-e} among entropy, gas density and gas pressure,  the approximate model \eqref{eq1.1simplified}--\eqref{eq1.3simplified} reduces
 \begin{subequations}\label{eq1.1simplified-red}
\begin{empheq}[right=\empheqrbrace\text{in $\R^3\setminus \Om(t)$, $t>0$,}]{align}
\pd_t {\bf v}_l + {\bf v}_l\cdot\nb{\bf v}_l =&\, \frac{\mu_l}{\rho_l}\De{\bf v}_l - \dfrac1{\rho_l}\, \nb p_l, \label{eq1.1simplified-red-a}\\
\div {\bf v}_l =&\, 0, \label{eq1.1simplified-red-b}
\end{empheq}
\end{subequations}

\begin{subequations}\label{eq1.2simplified-red}
\begin{empheq}[right=\empheqrbrace\text{in $\Om(t)$, $t>0$.}]{align}
&\pd_t \rho_g + \div(\rho_g{\bf v}_g) = 0,\qquad p_g = p_g(t),\label{eq1.2simplified-red-a}\\
&\pd_t \rho_g = \frac{\ka}{\ga c_v} \De \log\rho_g - \frac{\ka}{\ga c_v} \frac{|\nb\rho_g|^2}{\rho_g^2} - {\bf v}_g\cdot\nb\rho_g + \frac{\pd_tp_g}{\ga p_g} \rho_g, \label{eq1.2simplified-red-c}
\end{empheq}
\end{subequations}
and
 \begin{subequations}\label{eq1.3simplified-red}
\begin{empheq}[right=\empheqrbrace\text{on $\pd\Om(t)$, $t>0$,}]{align}
{\bf v}_l(\boldsymbol{\om},t)\cdot\hat{\bf n} = {\bf v}_g(\boldsymbol{\om},t)\cdot\hat{\bf n} = \pd_t{\boldsymbol{\om}}\cdot\hat{\bf n}, \label{eq1.3simplified-red-a}\\
p_g \hat{\bf n} - p_l \hat{\bf n} + 2\mu_l\hat{\bf n}\cdot\mathbb{D}({\bf v}_l) = \si \hat{\bf n} (\nb_S\cdot \hat {\bf n}), \label{eq1.3simplified-red-b}\\
p_g = \Rg T_\infty\rho_g, \label{eq1.3simplified-red-c}
\end{empheq}
\end{subequations}
The system \eqref{eq1.1simplified-red}--\eqref{eq1.3simplified-red} is supplemented by 
the initial conditions
\EQ{
{\bf v}_l(\cdot,0),\, \rho_g(\cdot,0),\, \Om(0).
}
At spatial infinity, we require the following far-field conditions: 
\EQ{\label{eq-far-field-all}
\max_{|x|=r}\,{\bf v}_l(x,t) = O(r^{-2}),\qquad
\lim_{|x|\to\infty} \nb{\bf v}_l(x,t) = \mathbb O,\qquad
\lim_{|x|\to\infty} p_l(x,t) = p_{\infty,*},
}
where $p_{\infty,*}>0$ is the far-field pressure. The  far-field conditions \eqref{eq-far-field-all} are motivated the spherically symmetric equilibrium 
${\bf v}_l = (C/r^2)\,\hat{\bf r}$, $r=|x|$ and $\hat{\bf r} = x/|x|$.
Moreover, the spatial decay conditions \eqref{eq-far-field-all} are required for establishing the spherical symmetry of equilibrium bubble in \cite[Proposition 4.3.\,(1)]{LW-vbas2022}.

\begin{remark}\label{rmk-far-field} 
We remark on a property of solutions which exhibit more rapid spatial decay. 
Suppose, for example, that  ${\bf v}_l = o(|x|^{-2})$, so that  $\lim_{|x|\to\infty} |x|^2 |{\bf v}_l(x,t)| = 0$. Then, we claim that the bubble volume, $|\Omega(t)|$, 
is independent of time.   Indeed, we calculate
\[
0 = \int_{\R^3\setminus\Om(t)} \div{\bf v}_l\, dx 
= \lim_{r\to\infty} \int_{\pd B_r} {\bf v}_l\cdot\hat{\bf n}_{\pd B_r}\, dS - \int_{\pd\Om(t)} {\bf v}_l\cdot\hat{\bf n}\, dS
= - \int_{\pd\Om(t)} \pd_t{\boldsymbol{\om}}\cdot\hat{\bf n}\, dS
= - \frac{d}{dt} |\Om(t)|,
\]
This holds in addition to the conservation of the bubble mass; see \cite[Proposition 7.3]{LW-vbas2022}.
Note that, in the radial case, we have that $|\Om(t)| = (4\pi/3) R^3(t)$.
So, the conservation of the bubble volume implies $R(t) \equiv$ constant.
Thus, by \cite[(5.1c)]{LW-vbas2022}, we have $\rho_g(R,t) = \frac1{\Rg T_\infty}\bkt{p_{\infty,*} + \frac{2\si}{R}}$.
Differentiating the equation with respect to $t$, we have $\pd_t\rho_g(R,t) = 0$, which implies $\pd_tp_g(t) = \Rg T_\infty\pd_t\rho_g(R,t) = 0$.
Hence, the solution is an equilibrium.
\end{remark}

\section{Characterization of equilibrium bubble shapes in an irrotational fluid}
Consider  approximate model \eqref{eq1.1simplified-red}--\eqref{eq-far-field-all} in the irrotational framework, {\it i.e.} $\nabla\times{\bf v}_l=0$. In this setting,
we show that the surface tension alone (without the liquid viscosity $\mu_l$) is sufficient to constrain the shape of any equilibrium bubble to be  spherical.

Since $\nabla\times{\bf v}_l=0$, we may introduce liquid and gas velocity potentials: ${\bf v}_l = \nb\phi_l$ and ${\bf v}_g = \nb\phi_g$.
Then the approximate model \eqref{eq1.1simplified-red}--\eqref{eq-far-field-all} becomes
 \begin{subequations}\label{eq1.1simplified-red-irro}
\begin{empheq}[right=\empheqrbrace\text{in $\R^3\setminus \Om(t)$, $t>0$,}]{align}
\pd_t\phi_l + \frac{|\nb\phi_l|^2}2 =&\, -\frac{p_l-p_{\infty,*}}{\rho_l}, \label{eq1.1simplified-red-irro-a}\\
\De\phi_l =&\, 0,\quad 
\label{eq1.1simplified-red-irro-b}
\end{empheq}
\end{subequations}
\begin{subequations}\label{eq1.2simplified-red-irro}
\begin{empheq}[right=\empheqrbrace\text{in $\Om(t)$, $t>0$.}]{align}
&\pd_t \rho_g + \rho_g\De\phi_g + \nb\rho_g\cdot\nb\phi_g = 0,\qquad p_g = p_g(t),\label{eq1.2simplified-red-irro-a}\\
&\pd_t \rho_g = \frac{\ka}{\ga c_v} \De \log\rho_g - \frac{\ka}{\ga c_v} \frac{|\nb\rho_g|^2}{\rho_g^2} - \nb\phi_g\cdot\nb\rho_g + \frac{\pd_tp_g}{\ga p_g} \rho_g, \label{eq1.2simplified-red-irro-c}
\end{empheq}
\end{subequations}
and
 \begin{subequations}\label{eq1.3simplified-red-irro}
\begin{empheq}[right=\empheqrbrace\text{on $\pd\Om(t)$, $t>0$,}]{align}
\nb\phi_l(\boldsymbol{\om},t)\cdot\hat{\bf n} = \nb\phi_g(\boldsymbol{\om},t)\cdot\hat{\bf n} = \pd_t{\boldsymbol{\om}}\cdot\hat{\bf n}, \label{eq1.3simplified-red-irro-a}\\
p_g \hat{\bf n} - p_l \hat{\bf n} + 2\mu_l\hat{\bf n}\cdot[D^2\phi_l] = \si \hat{\bf n} (\nb_S\cdot \hat {\bf n}), \label{eq1.3simplified-red-irro-b}\\
p_g = \Rg T_\infty\rho_g, \label{eq1.3simplified-red-irro-c}
\end{empheq}
\end{subequations}
where $[D^2\phi_l]$ is the Hessian of $\phi_l$,
with the initial conditions
\EQ{
\phi_l(\cdot,0),\, \rho_g(\cdot,0),\, \Om(0)={\boldsymbol{\om}}(\mathbb S^2,t=0),
}
and the far-field conditions 
\EQ{\label{eq-far-field-all-irro}
\max_{|x|=r}\, \nb\phi_l(x,t) = O(r^{-2}),\qquad
\lim_{|x|\to\infty} D^2\phi_l(x,t) = \mathbb O,\qquad
\lim_{|x|\to\infty} p_l(x,t) = p_{\infty,*}.
}

As discussed in \cite[Proposition 4.3]{LW-vbas2022}, the approximate system \eqref{eq1.1simplified-red}--\eqref{eq-far-field-all} admits a family of spherically symmetric equilibria, parameterized by the bubble mass, $M$.
The system \eqref{eq1.1simplified-red-irro}-\eqref{eq-far-field-all-irro}, which arises from the irrotationality assumption  admits the same family of equilibria:\\
\begin{subequations}
\label{eq-equilibrium}
\begin{align}
\phi_{l,*} &= c_1,\qquad\qquad 
p_{l,*} = p_{\infty,*} ,\qquad\qquad
\Om_* = B_{_{R_*[M]}}, \label{eq-equilibrium-a}\\
\rho_{g,*}[M] &= \frac1{\Rg T_\infty} \bke{ p_{\infty,*} +\frac{2\si}{R_*[M]} },\quad
\phi_{g,*} = c_2,\quad 
p_{g,*}[M] = p_{\infty,*} + \frac{2\si}{R_*[M]},\label{eq-equilibrium-b}\\
T_{g,*} &= T_{l,*} = T_\infty,\quad
s_* = c_v\log\bke{(\Rg T_\infty)^\ga \bke{ p_{\infty,*} +\frac{2\si}{R_*[M]}}^{1-\ga}},\label{eq-equilibrium-c}
\end{align}
\end{subequations}
where $c_1$, $c_2$ are constants,
\EQ{\label{eq-mass-converve-ss}
M := \int_{\Om_*} \rho_{g,*}\, dx  > 0,
}
and $R_* = R_*[M]$ is the unique positive solution to the cubic equation
\EQ{\label{eq-cubic}
p_{\infty,*} R_*^3 + 2\si R_*^2 - \frac{3\Rg T_\infty M}{4\pi} = 0.
}

In \cite[Proposition 4.3]{LW-vbas2022}, we proved that the surface tension and the viscosity constrain any equilibrium bubble of the approximate model \eqref{eq1.1simplified-red}--\eqref{eq-far-field-all}  to be of spherical shape.
The following theorem shows that, in the irrotational framework, surface tension is sufficient to ensure that the equilibrium bubble is spherical; viscosity is not required for this conclusion.

\begin{theorem}[Characterization equilibrium bubbles]\label{thm-irro-equilib}
Assume $\si\neq0$ and $\mu_l\ge0$, allowing $\mu_l=0$.
Then any regular equilibrium solution of the reduced irrotational system \eqref{eq1.1simplified-red-irro}--\eqref{eq-far-field-all-irro} is uniquely determined by its bubble mass $M$ and is given by \eqref{eq-equilibrium}.
In particular, irrotational equilibrium bubbles of the approximate system \eqref{eq1.1simplified-red}--\eqref{eq-far-field-all} are spherical.
\end{theorem}

\begin{remark}\label{rem:rotational-examples}
In contrast to Theorem \ref{thm-irro-equilib} we show, in a forthcoming article \cite{LW-aeb}, that if we allow flows with non-trivial vorticity then,  we can construct non-spherically symmetric equilibrium solutions of the system \eqref{eq1.1simplified-red}--\eqref{eq1.3simplified-red}.
\end{remark}

\begin{proof} Setting time-derivatives equal to zero, we obtain that
 steady-state solutions of  \eqref{eq1.1simplified-red-irro}--\eqref{eq-far-field-all-irro} solve
 \begin{subequations}\label{eq1.1simplified-equilib}
\begin{empheq}[right=\empheqrbrace\text{in $\R^3\setminus \Om_*$,}]{align}
\frac{|\nb\phi_{l,*}|^2}2 =&\, -\frac{p_{l,*} - p_{\infty,*}}{\rho_l}, \label{eq1.1simplified-a-equilib}\\
\De\phi_{l,*} =&\, 0, \label{eq1.1simplified-b-equilib}
\end{empheq}
\end{subequations}
Note that since $p_g=p_g(t)$, we have $p_{g*}$ is a constant.
 \begin{subequations}\label{eq1.2simplified-equilib}
\begin{empheq}[right=\empheqrbrace\text{in $\Om_*$,}]{align}
&\rho_{g,*}\De\phi_{g,*} + \nb\rho_{g,*}\cdot\nb\phi_{g,*} =0\, \label{eq1.2simplified-a-equilib}\\
& \frac{\ka}{\ga c_v}\De\log\rho_{g,*} - \frac{\ka}{\ga c_v} \frac{|\nb\rho_{g,*}|^2}{\rho_{g,*}^2} - \nb\phi_{g,*}\cdot\nb\rho_{g,*} = 0, \label{eq1.2simplified-b-equilib}
\end{empheq}
\end{subequations}
with interface conditions:
 \begin{subequations}\label{eq1.3simplified-equilib}
\begin{empheq}[right=\empheqrbrace\text{on $\pd\Om_*$,}]{align}
\nb\phi_{l,*}\cdot\hat{\bf n} = \nb\phi_{g,*}\cdot\hat{\bf n} = 0, \label{eq1.3simplified-a-equilib}\\
p_{g,*} \hat{\bf n} - p_{l,*} \hat{\bf n} + 2\mu_l \hat{\bf n} \cdot[D^2\phi_{l,*}] = \si \hat{\bf n} (\nb_S\cdot \hat {\bf n}), \label{eq1.3simplified-b-equilib}\\
p_{g,*} = \Rg T_\infty\rho_{g,*} \label{eq1.3simplified-c-equilib}
\end{empheq}
\end{subequations}
and the far-field velocity and pressure conditions:
\EQ{\label{eq-far-field-pressure-equilib}
\max_{|x|=r}\, \big|\nb\phi_{l,*}(x)\big| = O(r^{-2}),\qquad
\lim_{|x|\to\infty} D^2\phi_{l,*}(x) = \mathbb{O},\qquad
\lim_{|x|\to\infty} p_{l,*}(x) = p_{\infty,*}.
}

Let $\td\phi_{l,*} = \phi_{l,*} + {\bf c}$, where ${\bf c}$ is some constant vector to be determined.
Then $\td\phi_{l,*}$ is harmonic and satisfies  \eqref{eq1.3simplified-a-equilib}, \eqref{eq-far-field-pressure-equilib}.
Thus, \EQN{
0 &= \int_{\R^3\setminus\Om_*} \De\td\phi_{l,*}\cdot\td\phi_{l,*}\, dx \\
&= - \int_{\R^3\setminus\Om_*} |\nb\td\phi_{l,*}|^2\, dx + \lim_{r\to\infty} \int_{\pd B_r} \td\phi_{l,*}\nb\td\phi_{l,*}\cdot\hat{\bf n}_{\pd B_r}\, dS - \int_{\pd\Om_*} \td\phi_{l,*}\nb\td\phi_{l,*}\cdot\hat{\bf n}\, dS
}
where 
$\hat{\bf n}_{\pd B_r}$ and $\hat{\bf n}$ are the unit outward normals on $\pd B_r$ and $\Om_*$, respectively, and the last term vanishes since $\td\phi_{l,*}$ satisfies \eqref{eq1.3simplified-a-equilib}.
Therefore,
\EQ{\label{eq-equib-energy-tilde}
\int_{\R^3\setminus\Om_*} |\nb\td\phi_{l,*}|^2\, dx = \lim_{r\to\infty} \int_{\pd B_r} \td\phi_{l,*}\nb\td\phi_{l,*}\cdot\hat{\bf n}_{\pd B_r}\, dS.  
}

We claim that, for an appropriate choice of ${\bf c}$, right hand side of \eqref{eq-equib-energy-tilde} vanishes and hence $\nb\td\phi_{l,*}\equiv{\bf 0}$ on $\mathbb{R}^3\setminus\Om_*$.
  Our goal is to prove $\nb\td\phi_{l,*}\equiv{\bf 0}$, with a suitable choice of ${\bf c}$, by showing that the limit on the right hand side is zero.

We now discuss our choice of ${\bf c}$.
Fix $r_0>0$ such that $\overline{\Om_*}\subset\subset B_{r_0}(0)$. For example,  in spherical coordinates $(r,\th,\varphi)\in\R^3\setminus\Om_*$ for all $r\ge r_0$. 
For any $r\ge r_0$,
\[
\phi_{l,*}(r,0,0) = \phi_{l,*}(r_0,0,0) + \int_{r_0}^r \pd_{r'}\phi_{l,*}(r',0,0)\, dr' .
\]
By \eqref{eq-far-field-pressure-equilib}, $\abs{\pd_{r'}\phi_{l,*}(r',0,0)}\le r'^{-2}\in L^1_{r'}(r_0,\infty)$. It follows that  $\int_{r_0}^\infty \pd_{r'}\phi_{l,*}(r',0,0)\, dr'$  has a limit and hence
 the limit $\lim_{r\to\infty} \phi_{l,*}(r,0,0)$ exists.
Define ${\bf c} = \lim_{r\to\infty} \phi_{l,*}(r,0,0)$. Hence, $\td\phi_{l,*}(r,0,0)\to0$ as $r\to\infty$.

Next, note that from \eqref{eq-equib-energy-tilde} we have
\begin{align}\label{eq-equib-energy-tilde-upper}
\int_{\R^3\setminus\Om_*} |\nb\td\phi_{l,*}|^2\, dx &\le \limsup_{r\to\infty} \int_{\pd B_r} |\td\phi_{l,*}|\ |\nb\td\phi_{l,*}| dS\\
&\le  \limsup_{r\to\infty} \max_{|x|=r} |\td\phi_{l,*}|\times  \limsup_{r\to\infty} \max_{|x|=r} \left( r^2\ |\nb\td\phi_{l,*}|\right)\nonumber
\end{align}
The latter factor is bounded by the far-field bound \eqref{eq-far-field-pressure-equilib}, and so it suffices to prove that \[\lim_{r\to\infty} \max_{|x|=r}\ |\td\phi_{l,*}(x)| = 0.\]

 In view of  $\nb\td\phi_{l,*} = \pd_r\td\phi_{l,*}\hat{\bf r} + \frac1r\pd_\th\td\phi_{l,*}\hat{\boldsymbol{\th}} + \frac1{r\sin\th}\pd_\varphi\td\phi_{l,*}\hat{\boldsymbol{\varphi}}$ and 
 \eqref{eq-far-field-pressure-equilib}, we have for some $C>0$, independent of $\th$ and $\varphi$, that
 \begin{subequations}
\begin{empheq}{align}
&\abs{\pd_r\td\phi_{l,*}(r,\th,\varphi)}\le\frac{C}{r^2},\label{eq-pdr-phil-decay}\\
&\abs{\frac1r\pd_\th\td\phi_{l,*}(r,\th,\varphi)} \le \frac{C}{r^2},\text{ which implies } \abs{\pd_\th\td\phi_{l,*}(r,\th,\varphi)} \le\frac{C}r,\ \text{ and, }\label{eq-pdth-phil-decay}\\
&\abs{\frac1{r\sin\th}\pd_\varphi\td\phi_{l,*}(r,\th,\varphi)} \le \frac{C}{r^2},\text{ which implies } \abs{\pd_\varphi\td\phi_{l,*}(r,\th,\varphi)}\le\frac{C}r.\label{eq-pdphi-phil-decay}
\end{empheq}
\end{subequations}
Using the fundamental theorem for linear integral,
\[
\td\phi_{l,*}(r,\th,\varphi) = \td\phi_{l,*}(r,0,0) + \int_{\mathcal{C}} \nb_{\th,\varphi}\td\phi_{l,*}(r,\th',\varphi')\cdot d\ell_{\th',\varphi'},
\]
where $\mathcal{C}$ is a curve of finite length connecting $(0,0)$ and $(\th,\varphi)$.
Thus, by using \eqref{eq-pdth-phil-decay}, \eqref{eq-pdphi-phil-decay} and the fact that $\lim_{r\to\infty}\td\phi_{l,*}(r,0,0) = 0$, we have 
\[
\abs{ \td\phi_{l,*}(r,\th,\varphi) } \le \abs{ \td\phi_{l,*}(r,0,0)} + \frac{C \cdot \text{length}(\mathcal{C}) }r\to0,\ \text{ as }r\to\infty.
\]
Consequently, we have $\lim_{|x|\to\infty} \td\phi_{l,*}(x) = 0$ so that the limit on the right hand side of \eqref{eq-equib-energy-tilde} is zero, which implies $\nb\phi_{l,*} = \nb\td\phi_{l,*}\equiv{\bf 0}$
on $\mathbb R^3\setminus \Om_*$. 

Since $\nb\phi_{l,*} = {\bf 0}$, \eqref{eq1.1simplified-a-equilib} implies $p_{l,*} = p_{\infty,*}$ is constant.
Moreover, since $[D^2\phi_{l,*}]=\mathbb{O}$, the stress balance equation \eqref{eq1.3simplified-b-equilib} becomes
\[
p_{g,*} - p_{l,*} = \si (\nb_S\cdot \hat {\bf n})\ \text{ on }\pd\Om_*.
\] 
Since $p_{g,*}$ and $p_{l,*}$ are both constant, $\pd\Om_*$ is a closed constant-mean-curvature surface.
By Alexandrov's Theorem \cite{Alexandrov-AMPA1962}, $\Om_*$ must be a sphere.

We now deal with the system \eqref{eq1.2simplified-equilib} for the gas.
Multiplying \eqref{eq1.2simplified-a-equilib} by $\phi_{g,*}$ and then integrating the equation over $\Om_*$, we obtain
\EQN{
0 = \int_{\Om_*} \rho_{g,*} (\De\phi_{g,*}) \phi_{g,*}\, dx + \int_{\Om_*} (\nb\rho_{g,*}\cdot\nb\phi_{g,*})\phi_{g,*}\, dx
= - \int_{\Om_*} \rho_{g,*}|\nb\phi_{g,*}|^2\, dx,
}
where we've used the integration by parts formula and the boundary condition \eqref{eq1.3simplified-a-equilib}.
This implies that $\nb\phi_{g,*} \equiv {\bf 0}$.
Thus, by \eqref{eq1.2simplified-b-equilib}, we have 
\EQ{\label{eq1.2simplified-b-equilib-pf}
0 = \De\log\rho_{g,*} - \frac{|\nb\rho_{g,*}|^2}{\rho_{g,*}^2} = -\rho_{g,*}\De\bke{\frac1{\rho_{g,*}}},
}
which yields that $\De\bke{\frac1{\rho_{g,*}}} = 0$.
Hence, since $\rho_{g,*}|_{\pd\Om_*}$ is a constant by \eqref{eq1.3simplified-c-equilib},
\[
0 = \int_{\Om_*} \De\bke{\frac1{\rho_{g,*}}}\, dx = \int_{\pd\Om_*} \nb\bke{\frac1{\rho_{g,*}}}\cdot\hat{\bf n}\, dS = -\frac1{ (\rho_{g,*} |_{\pd\Om_*})^2} \int_{\pd\Om_*} \nb\rho_{g,*}\cdot\hat{\bf n}\, dS,
\]
implying $\int_{\pd\Om_*} \nb\rho_{g,*}\cdot\hat{\bf n}\, dS = 0$.
Therefore, integrating \eqref{eq1.2simplified-b-equilib-pf} over $\Om_*$ implies
\EQN{
0 &= \int_{\Om_*} \De\log\rho_{g,*}\, dx - \int_{\Om_*} \frac{|\nb\rho_{g,*}|^2}{\rho_{g,*}^2}\, dx
= \int_{\pd\Om_*} \nb\log\rho_{g,*}\cdot\hat{\bf n}\, dS - \int_{\Om_*} \frac{|\nb\rho_{g,*}|^2}{\rho_{g,*}^2}\, dx\\
&= \frac1{\rho_{g,*}|_{\pd\Om_*}} \int_{\pd\Om_*} \nb\rho_{g,*}\cdot\hat{\bf n}\, dS - \int_{\Om_*} \frac{|\nb\rho_{g,*}|^2}{\rho_{g,*}^2}\, dx
= - \int_{\Om_*} \frac{|\nb\rho_{g,*}|^2}{\rho_{g,*}^2}\, dx,
}
which implies $\rho_{g,*}=$ constant. 
As discussed in the proof of \cite[Proposition 4.3]{LW-vbas2022}, the constant $\rho_{g,*}$ is given by \cite[(4.14a)]{LW-vbas2022} and is determined by the equilibrium radius $R_*$ and the bubble mass $M$, where $R_* = R_*[M]$ is the unique positive solution to the cubic equation \eqref{eq-cubic}.
This completes the proof of the theorem.
\end{proof}

\section{Linearized perturbation dynamics about a spherically symmetric equilibrium}
\label{sec:linearized-dynamics}

In this section we derive the linearized evolution equations which govern infinitesimal perturbations of an equilibrium.
 We expand \eqref{eq1.1simplified-red-irro}--\eqref{eq-far-field-all-irro} around any fixed equilibrium in \eqref{eq-equilibrium}: 
\EQ{\label{eq-linearize-old}
\phi_l &= c_1 + \de \Phi_l + O(\de^2),\qquad\quad
p_l = p_{\infty,*} + \de \mathcal P_l + O(\de^2),\\
\boldsymbol{\om}(\th,\varphi,t) 
&= R(\th,\varphi,t) \hat{\bf r} 
= \bkt{R_* + \de \mathcal R(\th,\varphi,t) + O(\de^2) } \hat{\bf r},\quad |\hat{\bf r}|=1, \\
\hat{\bf n} &= \hat{\bf r} + O(\de),\qquad
\nb_S\cdot\hat{\bf n} = \frac2{R_*} - \de \frac1{R_*^2} (2+\De_S)\mathcal R + O(\de^2)\quad \text{(see e.g. \cite[(C.27)]{SW-SIMA2011})},
\\
\rho_g &= \rho_* + \de\varrho + O(\de^2),\qquad\quad
\phi_g = c_2 + \de \Phi_g + O(\de^2),\qquad\quad
p_g = p_* + \de \mathcal P_g + O(\de^2).
}
%
Retaining only terms which are of $O(\de)$ and making the change of the variables $x=R_*y$, we derive the linearized system
 for 
 \[ \Phi_l(y,t),\ \mathcal{P}_l(y,t),\ \Phi_g(y,t),\ \varrho(y,t),\ \mathcal{P}_g(t):\]
 \begin{subequations}\label{eq1.1modified-lin-irro-0}
\begin{empheq}[right=\empheqrbrace\text{in $\R^3\setminus B_1$, $t>0$,}]{align}
\pd_t \Phi_l =& - \dfrac1{\rho_*R_*}\, \mathcal P_l, \label{eq1.1modified-a-lin-irro-0}\\
\De_y\Phi_l =&\, 0,\quad \label{eq1.1modified-b-lin-irro-0}
\end{empheq}
\end{subequations}

 \begin{subequations}
\begin{empheq}[right=\empheqrbrace\text{in $B_1$, $t>0$.}]{align}
\pd_t \varrho + \frac{\rho_*}{R_*} \De_y\Phi_g =&\, 0,\qquad \mathcal{P}_g = \mathcal{P}_g(t),\label{eq1.2modified-a-lin-irro-0}\\
\pd_t \varrho =&\, \dfrac{\ka}{\ga c_v} \dfrac1{\rho_*R_*^2} \De_y \varrho + \dfrac{\rho_*}{\ga p_*}\pd_t\mathcal P_g\label{eq1.2modified-b-lin-irro-0}, 
\end{empheq}
\end{subequations}

For the boundary conditions \eqref{eq1.3simplified-red-irro}, it is clear that \eqref{eq1.3simplified-red-irro-a} and \eqref{eq1.3simplified-red-irro-c} are linearized to $\pd_r\Phi_l = \pd_r\Phi_g = \pd_t\mathcal R$ and $\mathcal P_g = \Rg T_\infty \varrho$, respectively.
For the stress balance equation \eqref{eq1.3simplified-red-irro-b}, we obtain the linearized stress balance equation
\EQN{
\mathcal P_g\hat{\bf r} - \mathcal P_l\hat{\bf r} + 2\mu_l\hat{\bf r}\cdot[D^2_y\Phi_l] = -\frac{\si}{R_*^2}(2+\De_S)\mathcal R\hat{\bf r},\qquad \text{ on } \pd B_1.
}
We compute
\EQN{
\hat{\bf r}\cdot[D_y^2\Phi_l] &= \hat{\bf r}\cdot\nb_y (\nb_y\Phi_l) 
= \pd_r\bke{\nb_y\Phi_l}
= \pd_r\bke{\pd_r\Phi_l\, \hat{\bf r} + \frac1r\pd_\th\Phi_l\, \hat{\boldsymbol{\th}} + \frac1{r\sin\th}\pd_\varphi\Phi_l\, \hat{\boldsymbol{\varphi}} }\\
&= \pd_r^2\Phi_l\, \hat{\bf r} + \bke{-\frac1{r^2}\pd_\th\Phi_l + \frac1r\pd_r\pd_\th\Phi_l} \hat{\boldsymbol{\th}} + \bke{-\frac1{r^2\sin\th}\pd_\varphi\Phi_l + \frac1{r\sin\th}\pd_r\pd_\varphi\Phi_l} \hat{\boldsymbol{\varphi}}.
}
Then the radial component of the linearized stress balance equation reads
\EQN{
\mathcal P_g(t) - \mathcal P_l\big|_{\pd B_1} + \frac{2\mu_l}{R_*^2} \pd_r^2\Phi_l\big|_{\pd B_1} = -\frac{\si}{R_*^2}(2+\De_S)\mathcal R,
}
and, if $\mu_l>0$, the tangential ($\hat{\boldsymbol{\th}}$ and $\hat{\boldsymbol{\varphi}}$) components of the equation yields
\[
\pd_\th\Phi_l = \pd_r\pd_\th\Phi_l,\qquad
\pd_\varphi\Phi_l = \pd_r\pd_\varphi\Phi_l.
\]
Therefore, we have
 \begin{subequations}\label{eq1.3modified-lin-irro-0}
\begin{empheq}[right=\empheqrbrace\text{on $\pd B_1$, $t>0$,}]{align}
\pd_r\Phi_l =&\, \pd_r\Phi_g = \pd_t\mathcal R, \label{eq1.3modified-a-lin-irro-0}\\
\mathcal P_g - \mathcal P_l +\frac{2\mu_l}{R_*^2}\pd_r^2\Phi_l =& -\frac{\si}{R_*^2}(2+\De_S)\mathcal R, \label{eq1.3modified-b-lin-irro-0}\\
 \pd_\th\pd_r\Phi_l - \pd_\th\Phi_l =&\, 0,\ \text{ if $\mu_l>0$,} \label{eq1.3modified-b1-lin-irro-1}\\
 \pd_\varphi\pd_r\Phi_l - \pd_\varphi\Phi_l =&\, 0,\ \text{ if $\mu_l>0$,} \label{eq1.3modified-b2-lin-irro-1}\\
\mathcal P_g =&\, \Rg T_\infty \varrho,\label{eq1.3modified-c-lin-irro-0}
\end{empheq}
\end{subequations}
with the far-field conditions 
\EQ{\label{eq-far-field-0}
\nb\Phi_l(y,t) = O(|y|^{-2}),\qquad
D^2\Phi_l(y,t) \to \mathbb O,\qquad
\mathcal P_l(y,t) \to 0,\quad \textrm{as}\quad  |y|\to\infty.
}
Note that, by applying the Laplacian to \eqref{eq1.1modified-a-lin-irro-0} and then using \eqref{eq1.1modified-b-lin-irro-0}, we have that $\mathcal{P}_l$ solves the following boundary value problem on the exterior of $B_1$:
\EQN{
\De_y\mathcal{P}_l &= 0,\qquad\qquad\qquad\qquad\qquad\qquad\qquad \, \text{ in }\R^3\setminus B_1,\\
\mathcal{P}_l &= \mathcal{P}_g + \frac{2\mu_l}{R_*^2} \pd_r^2\Phi_l + \frac{\si}{R_*^2}(2+\De_S)\mathcal{R},\quad\, \text{ on }\pd B_1,\\
\mathcal{P}_l&\to0,\qquad\qquad\qquad\qquad\qquad\qquad\qquad  \text{ as }|y|\to\infty.
}
Therefore, $\mathcal{P}_l = \mathcal{P}_l[\mathcal{R}, \Phi_l, \mathcal{P}_g]$ by using the Poisson kernel (normal derivative of Dirichlet Green function) for the Laplace equation on the exterior domain $\R^3\setminus B_1$.  
Moreover, since $\mathcal{P}_g = \mathcal{P}_g(t)$, $\mathcal{P}_g = \mathcal{P}_g[\varrho]$ by \eqref{eq1.3modified-c-lin-irro-0}.
Furthermore, since $\Phi_g$ satisfies the Poisson's equation with non-homogeneous Neumann:
\EQ{\label{eq-Phi-g-Neumann}
\De_y\Phi_g &= - \frac{R_*}{\rho_*} \pd_t\varrho,\quad \text{ in $B_1$, $t>0$},\\
\pd_r\Phi_g &= \pd_t\mathcal{R},\qquad\quad \text{ on $\pd B_1$, $t>0$},
}
we have $\Phi_g = \Phi_g[\mathcal{R}, \varrho]$.
Note that one can solve \eqref{eq1.1modified-lin-irro-0}, \eqref{eq1.2modified-b-lin-irro-0}, \eqref{eq1.3modified-lin-irro-0}, \eqref{eq-far-field-0} for $(\mathcal{R}, \Phi_l,\varrho)$ without using $\Phi_g$ (see Section \ref{sec-linear-wellposed} below). 
Therefore, \eqref{eq1.1modified-lin-irro-0}--\eqref{eq-far-field-0} can be reduced to a problem with unknowns 
\[
\mathcal{R}(\th,\varphi,t),\ \Phi_l(y,t),\ \varrho(y,t).
\]
This completes our formulation of the linearized dynamics.

By choosing the origin of coordinates to be  the centroid of volume of the bubble, we may assume the linearized zero-centroid-of-volume condition 
\EQ{\label{eq-linear-centroid}
\bka{\mathcal R, Y_1^m}_{L^2(\mathbb S^2)} = 0,\qquad m=-1,0,1;
}
see Appendix \ref{appdx-centroid-frame}.

\subsection{Initial data}
The linearized system \eqref{eq1.1modified-lin-irro-0}--\eqref{eq-far-field-0} is coupled with the initial data 
\[
\mathcal{R}(\cdot,\cdot,0),\
\Phi_l(\cdot,0),\
\varrho(\cdot,0).
\]

To investigate the asymmetric deformations of a nearly spherical gas bubble, 
we fix an equilibrium $(\rho_{g,*}, B_{R_*},\ldots)$ and choose the initial data to be a perturbation of the equilibrium.
For the isobaric approximation, the mass inside the bubble is independent of time; there is no mass exchange across the bubble-fluid boundary; see \cite[Proposition 7.3]{LW-vbas2022}, 
if the initial data has different mass than the equilibrium, one should not expect the solution converging to the equilibrium.
Indeed, the solution is expected to converge to a near-by equilibrium that has same mass as the initial data.
In view of the continuity result in \cite[Proposition 4.7]{LW-vbas2022}, we may consider perturbations that have the same mass as the equilibrium $(\rho_{g,*}, B_{R_*},\ldots)$. 
In light of the linearized conservation of mass \cite[(2.3)]{LW-vbaslinear2023}, this mass constraint is equivalent to imposing the following constraint on the initial data  \cite[(2.5)]{LW-vbaslinear2023}:
\EQ{\label{eq-linear-mass-conserve}
\int_{B_1} \varrho(y,0)\, dy + 4\pi\frac{\rho_*}{R_*}\left\langle Y_0^0,\mathcal {R}(\cdot,\cdot,0)\right\rangle_{_{L^2(\mathbb{S}^2)}} = 0.
}

\subsection{Main Theorems}

In this subsection we state the main results of the present paper.

We develop the well-posedness theorem for the inviscid linearized system in the following function space:
\EQ{\label{eq-X}
X = \Bigg\{(\mathcal{R}, \Phi_l,\varrho)&: \mathcal{R}\in L^2(\mathbb{S}^2),\, 
\Phi_l\in L^2(\R^3\setminus B_1)\ \text{satisfying $\De\Phi_l=0$ and the decay \eqref{eq-far-field-0},}\\
&\quad \varrho\in L^2(B_1)\ \text{satisfying the constraint \eqref{eq-linear-mass-conserve}}\Bigg\}.
}

Let $Y_\ell^m:\mathbb{S}^2\to\R$ be the spherical harmonic function of degree $\ell$ and order $m$, where $\ell=0,1,2,\ldots$, and $|m|\le\ell$.  
It satisfies $-\De_{\mathbb S^2}Y_\ell^m = \ell(\ell+1)Y_\ell^m$.

\begin{theorem}[Linear well-posedness for the inviscid case]\label{thm-wellposed}
If $\mu_l = 0$, then for any initial data $(\mathcal{R}(\cdot,0), \Phi_l(\cdot,0),\varrho(\cdot,0))\in X$ there exists a unique solution $(\mathcal{R},\Phi_l,\varrho)\in C^1([0,\infty);X)$ of the linearized system \eqref{eq1.1modified-lin-irro-0}--\eqref{eq-linear-mass-conserve}.

Moreover, the solution exhibits the following properties in the monopole / radial mode and the multipole / shape modes:

\textit{\textbf{Monopole / radial mode}}:
\EQ{\label{eq-radial-mode}
\bka{Y_0^0,\mathcal{R}}_{L^2(\mathbb{S}^2)}(t),\ \sup_{r\ge1}\, \bka{Y_0^0,\Phi_l(r,\cdot,\cdot,t)}_{L^2(\mathbb{S}^2)}(t),\text{ and }&\int_{B_1} \abs{\bka{Y_0^0, \varrho(|y|,\cdot,\cdot,t)}_{L^2(\mathbb{S}^2)}}^2 dy\\
& \text{decay exponentially to zero as $t\to\infty$,}
}
and if, additionally, $\bka{Y_0^0, \varrho(|y|,\cdot,\cdot,0)}_{L^2(\mathbb{S}^2)}\in C^{2+2\al}_y(B_1)$ for some $\al\in(0,1/2)$, then
\EQ{\label{eq-radial-mode-Phi-g}
\sup_{r\le1}\, \bka{Y_0^0, \Phi_g(r,\cdot,\cdot,t)}_{L^2(\mathbb{S}^2)}\ \text{ and }&\ \sup_{r\le1}\, \bka{Y_0^0, \nb\Phi_g(r,\cdot,\cdot,t)}_{L^2(\mathbb{S}^2)}\\
&\qquad \text{ decay exponentially to zero as $t\to\infty$.}
}

\textit{\textbf{Multipole / shape modes}}: For $\ell\ge2$, 
\EQ{\label{eq-shape-modes}
\bka{Y_\ell^m,\mathcal{R}}_{L^2(\mathbb{S}^2)}(t)\ \text{ and }\ &\bka{Y_\ell^m,\Phi_l(r,\cdot,\cdot,t)}_{L^2(\mathbb{S}^2)}\ \text{ are }\\
&\qquad \text{$(2\pi)/\sqrt{\frac{\si}{\rho_lR_*^3}(\ell+2)(\ell+1)(\ell-1)}$-periodic functions in $t$},
}
\EQ{\label{eq-shape-modes-rho}
\int_{B_1} \abs{ \bka{Y_\ell^m,\varrho(|y|,\cdot,\cdot,t)}_{L^2(\mathbb{S}^2)}}^2 dy\ \text{ decays exponentially to zero as $t\to\infty$, }
}
and
\EQ{\label{eq-shape-modes-Phi-g}
\bka{Y_\ell^m, \Phi_g(r,\cdot,\cdot,t)}_{L^2(\mathbb{S}^2)}\ \text{ and }\ &\bka{Y_\ell^m, \nb\Phi_g(r,\cdot,\cdot,t)}_{L^2(\mathbb{S}^2)}\ \text{ are asymptotically }\\
&\text{$(2\pi)/\sqrt{\frac{\si}{\rho_lR_*^3}(\ell+2)(\ell+1)(\ell-1)}$-periodic functions in $t$}.
}

\end{theorem}

The proof of Theorem \ref{thm-wellposed} is  presented in Section \ref{sec-linear-wellposed}.

The following theorem demonstrates the incompatibility between the viscosity and fluid irrotationality in the linearized system \eqref{eq1.1modified-lin-irro-0}--\eqref{eq-far-field-0}.

\begin{theorem}[Linear ill-posedness of the viscous irrotational problem IVP]\label{thm-linear-illposed} 
Assume $\mu_l>0$.
Assume $(\mathcal{R},\Phi_l,\varrho)\in C^1([0,\infty);X)$, continuous up to $t=0$, is a solution to the linearized system \eqref{eq1.1modified-lin-irro-0}--\eqref{eq-linear-mass-conserve}.
Then $\Phi_l= \Phi_l(r,t)$ and $\mathcal{R}= \mathcal{R}(t)$ are independent of $\th$ and $\varphi$ for all $t\ge0$.
\end{theorem}

In other words, any regular solution  of the linearized system \eqref{eq1.1modified-lin-irro-0}--\eqref{eq-linear-mass-conserve}, which attains its initial values continuously, can only arise from radially symmetric initial data. Note also that a  radial velocity field is necessarily irrotational, so there is no contradiction with the well-posedness of the radial viscous problem. Theorem \ref{thm-linear-illposed} is proved in Section \ref{sec:linear-illposed}.

Based on the linear ill-posedness above (Theorem \ref{thm-linear-illposed}), we establish the following ill-posedness for the nonlinear, viscous irrotational problem \eqref{eq1.1simplified-red-irro}--\eqref{eq-far-field-all-irro}.

\begin{theorem}[Nonlinear ill-posedness for the viscous case]\label{thm-irro-illposed}
If $\mu_l>0$, then the irrotational system \eqref{eq1.1simplified-red-irro}--\eqref{eq-far-field-all-irro} is ill-posed in the sense that described in Section \ref{sec-illposed}.
\end{theorem}

The proof of Theorem \ref{thm-irro-illposed} is presented  in Section \ref{sec-illposed}.

\section{Spherical harmonic analysis of the non-viscous linearized system; proof of Theorem \ref{thm-wellposed} }\label{sec-linear-wellposed}

In view of Theorem \ref{thm-linear-illposed}  and the comment that follows it, we focus on the inviscid  ($\mu_l=0$)  IVP for general data, 
 and discuss the 
large time behavior of the solution.


We represent the state variables $\mathcal{R}, \Phi_l, \varrho, \Phi_g$ with respect to a  spherical harmonic basis $\{Y_\ell^m(\theta,\varphi)\}$ for  $L^2(\mathbb{S}^2)$, satisfying  $-\De_{\mathbb S^	2}Y_\ell^m = \ell(\ell+1)Y_\ell^m$:
\begin{align}
\label{eq-decompose-R}
\mathcal R(\th,\varphi,t) &= \sum_{\ell=0}^\infty \sum_{|m|\le\ell} a_\ell^m(t) Y_\ell^m(\th,\varphi),
\quad a_l^m(t)=\left\langle Y_\ell^m,\mathcal R\right\rangle_{_{L^2(\mathbb{S}^2)}}(t),\\
\label{eq-decompose-Phi-l}
\Phi_l(r,\th,\varphi,t) &= \sum_{\ell=0}^\infty \sum_{|m|\le\ell} b_\ell^m(t) \bke{\frac1r}^{\ell+1} Y_\ell^m(\th,\varphi),\quad b_l^m(t)=\left\langle Y_\ell^m, \Phi_l\big|_{r=1}\right\rangle_{_{L^2(\mathbb{S}^2)}}(t),\\
\label{eq-rho-deomposition}
\varrho(r,\th,\varphi,t) &= \sum_{\ell=0}^\infty \sum_{|m|\le\ell} f^m_\ell(r,t)Y_\ell^m(\th,\varphi)
=: \sum_{\ell=0}^\infty \sum_{|m|\le\ell} \varrho_\ell^m(r,\th,\varphi,t),\\
  f_l^m(r,t) &=\left\langle Y_\ell^m, \varrho(r,\cdot,\cdot,t)\right\rangle_{_{L^2(\mathbb{S}^2)}}(r,t), \nonumber\\
\label{eq-Phig-decomposition}
\Phi_g(r,\th,\varphi,t) &= \sum_{\ell=0}^\infty \sum_{|\ell|\le m} \Psi_\ell^m(r,t) Y_\ell^m(\th,\varphi)
=: \sum_{\ell=0}^\infty \sum_{|\ell|\le m} (\Phi_g)_\ell^m(r,\th,\varphi,t),\\
\Psi_\ell^m(r,t) &=\left\langle Y_\ell^m, \Phi_g(r,\cdot,\cdot,t)\right\rangle_{_{L^2(\mathbb{S}^2)}}(r,t). \nonumber
\end{align}
The expression \eqref{eq-decompose-Phi-l} is a multipole expansion 
\cite{Thorne-RMP1980}; each term is harmonic (hence  $\De\Phi_l=0$ in $\R^3\setminus B_1$) and satisfies the far-field conditions \eqref{eq-far-field-0}.
 The expansion will be constructed and convergence issues will be addressed below.

%
%
%
%
%
%
%
%
%
%

\subsection{Observing the bubble from the outside liquid}\label{sec-outside}

\begin{proposition}\label{prop-liquid-ode}
Let $(\mathcal{R}, \Phi_l,\mathcal P_l)$ be the solution of the system \eqref{eq1.1modified-lin-irro-0} with the boundary condition \eqref{eq1.3modified-a-lin-irro-0} and the far-field conditions \eqref{eq-far-field-0}.
Suppose $\mathcal R$ and $\Phi$ have the decompositions \eqref{eq-decompose-R} and \eqref{eq-decompose-Phi-l}. In particular,
\[ a_l^m(t)=\left\langle Y_\ell^m,\mathcal R\right\rangle_{_{L^2(\mathbb{S}^2)}}(t),\quad \textrm{and}\quad 
b_l^m(t)=\left\langle Y_\ell^m, \Phi_l\big|_{r=1}\right\rangle_{_{L^2(\mathbb{S}^2)}}(t).
\]
Then, 
\begin{align} \label{eq-dot-a}
\dot a_\ell^m(t) &= - (\ell+1) b_\ell^m(t),\qquad \ell=0,1,2,\ldots,\quad |m|\le\ell, 
\\ \label{eq-Pl-expand}
\mathcal P_l(r,\th,\varphi,t) &= -\rho_lR_* \sum_{\ell=0}^\infty \sum_{|m|\le\ell} \dot b_\ell^m(t) \bke{\frac1r}^{\ell+1} Y_\ell^m(\th,\varphi).
\end{align}
\end{proposition}

\begin{proof}
Using the decompositions \eqref{eq-decompose-Phi-l} for $\Phi_l$ and \eqref{eq-decompose-R}  for $\mathcal{R}$, we have:
\begin{align*}
\partial_r\Phi_l(r,\theta,\varphi)\big|_{r=1} &=\sum_{\ell=0}^\infty \sum_{|m|\le\ell} b_\ell^m(t) [-(\ell+1)] Y_\ell^m(\th,\varphi),\quad {\rm and}\\
 \dot{\mathcal{R}}(\theta,\varphi,t), &= \sum_{\ell=0}^\infty \sum_{|m|\le\ell} \dot a_\ell^m(t)Y_\ell^m(\th,\varphi) 
\end{align*}
In view of the kinematic boundary condition \eqref{eq1.3modified-a-lin-irro-0}, equating these expressions yields  \eqref{eq-dot-a}.
The expression \eqref{eq-Pl-expand} follows from \eqref{eq1.1modified-a-lin-irro-0}.
This proves the proposition.
\end{proof}

We now consider the linearized Laplace--Young condition \eqref{eq1.3modified-b-lin-irro-0}. 
Evaluating \eqref{eq1.1modified-a-lin-irro-0} at $\pd B_1$ and using \eqref{eq1.3modified-b-lin-irro-0}, we get
\[
\pd_t\Phi_l \big|_{\pd B_1} = -\frac1{\rho_lR_*}\mathcal P_l \big|_{\pd B_1} = -\frac1{\rho_lR_*} \mathcal P_g(t) -\frac{\si}{\rho_l R_*^3}(2+\De_S)\mathcal R.
\]
Using the expansions \eqref{eq-decompose-R} and \eqref{eq-decompose-Phi-l} in the above equation gives
\[
\sum_{\ell=0}^\infty \sum_{|m|\le\ell} \dot b_\ell^mY_\ell^m 
= -\frac1{\rho_lR_*}\mathcal P_g(t)  + \frac{\si}{\rho_lR_*^3}  \sum_{\ell=0}^\infty \sum_{|m|\le\ell} a_\ell^m (\ell+2)(\ell-1)Y_\ell^m,
\]
where we've used the equation $-\De_{\mathbb{S}^2}Y_\ell^m = \ell(\ell+1)Y_\ell^m$.
Taking the $L^2(\mathbb{S}^2)$ inner product of the above equation with $Y_\ell^m$, $\ell=0,1,2,\ldots$, $|m|\le\ell$, we obtain the ordinary differential equations
 \begin{subequations}\label{eq-dot-b}
\begin{empheq}{align}
\dot b_0^0 &= - \frac{2\sqrt\pi}{\rho_lR_*} \mathcal P_g(t) - \frac{2\si}{\rho_lR_*^3}\, a_0^0,\label{eq-dot-b-a}\\
\dot b_\ell^m &= \frac{\si}{\rho_lR_*^3} (\ell+2)(\ell-1) a_\ell^m,\quad \ell\ge1.\label{eq-dot-b-b}
\end{empheq}\end{subequations}

Using \eqref{eq-dot-a} in \eqref{eq-dot-b}, we obtain that the monopole amplitude of $\mathcal{R}(\theta,\varphi,t)$, $a_0^0(t)$, satisfies the pressure-forced ODE:
\EQ{\label{eq-a-old-0mode}
\rho_lR_* \ddot a_0^0  - \frac{2\si}{R_*^2} a_0^0 = 2\sqrt\pi \mathcal P_g(t),
}
and that the shape modes amplitudes: $a_\ell^m(t)$, where $\ell\ge2$ and $|m|\le \ell$,  satisfy the equations: 
\EQ{\label{eq-a-old-hmode}
\rho_lR_* \ddot a_\ell^m + \frac{\si}{R_*^2} (\ell+2)(\ell+1)(\ell-1) a_\ell^m = 0,\quad \ell\ge2.
}

Clearly, 
for $\ell\ge2$, all solutions to the ODE \eqref{eq-a-old-hmode} merely oscillate; they exhibit no damping.
We note that this system
is equivalent to that derived in \cite[(6.5)]{Longuet-JFM1989} solved in \cite[(6.6)--(6.7)]{Longuet-JFM1989}:
\EQ{\label{eq-a-soln}
a_\ell^m(t) =  c_1\cos\bke{\sqrt{\frac{\si}{\rho_lR_*^3}(\ell+2)(\ell+1)(\ell-1)}\, t} + c_2\sin\bke{\sqrt{\frac{\si}{\rho_lR_*^3}(\ell+2)(\ell+1)(\ell-1)}\, t},
}
for some constants $c_1, c_2$.

It follows from \eqref{eq-dot-a} and \eqref{eq-a-soln} that, for $\ell\ge2$, $b_\ell^m(t)$ is a $(2\pi)/\sqrt{\frac{\si}{\rho_lR_*^3}(\ell+2)(\ell+1)(\ell-1)}$-periodic function. 
This proves \eqref{eq-shape-modes} by using the expression \eqref{eq-decompose-Phi-l}.

\medskip

While \eqref{eq-a-old-hmode} is a closed equation for the amplitudes $(a_\ell^m(t))$, where $\ell\ge2$ and $|m|\le \ell$, 
equation \eqref{eq-a-old-0mode} is coupled to the dynamics of the gas through the forcing term given by the gas pressure perturbation, $\mathcal P_g(t)$.

\subsection{Observing the bubble from the inside gas}\label{sec-inside}


We now study the linearized dynamics inside the gas, beginning with the radial or monopole ($\ell=0$) contribution and then turning to the shape / multipole modes $\ell\ge2$.

\subsubsection{Radial / monopole $(\ell=0)$ linearized dynamics inside the gas}\label{sec:radial-dynamics}

Recall that
\EQN{ 
a_0^0(t)&=\left\langle Y_0^0,\mathcal R\right\rangle_{_{L^2(\mathbb{S}^2)}}(t) = \frac{1}{2\sqrt\pi}\int_{\mathbb S^2} \mathcal R(\cdot,\cdot,t),\quad \textrm{and}\quad \\
f_0^0(r,t)&=\left\langle Y_0^0, \varrho(r,\cdot,\cdot,t)\right\rangle_{_{L^2(\mathbb{S}^2)}}(t)= \frac{1}{2\sqrt\pi}\int_{\mathbb S^2}\varrho(r,\cdot,\cdot,t).
}

\begin{proposition}\label{prop-l=0}
Suppose $(\mathcal{R},\Phi_l, \varrho)$ is a solution of \eqref{eq1.1modified-lin-irro-0}--\eqref{eq-linear-mass-conserve} and has the expansions \eqref{eq-decompose-R}--\eqref{eq-Phig-decomposition}.
Then $(f_0^0(r,t),a_0^0(t))$ solves the system
 \begin{subequations}\label{eq-f00}
\begin{empheq}{align}
\pd_tf_0^0(r,t) &= \frac{\ka}{\ga c_v} \frac1{\rho_*R_*^2} \De_rf_0^0(r,t) + \frac1{\ga} \pd_tf_0^0(1,t),\quad 0\le r\le 1,\ t>0,\label{eq-f00-a}\\
\int_{B_1} f_0^0(|y|,t)\, dy &= -4\pi\frac{\rho_*}{R_*} a_0^0(t),\quad t>0,\label{eq-f00-b}\\
f_0^0(1,t) &= \frac1{\Rg T_\infty} \bke{-\frac{2\si}{R_*^2} a_0^0(t) + \rho_lR_*\ddot{a}_0^0(t)},\quad t>0.\label{eq-f00-c}
\end{empheq}
\end{subequations}
where $\De_r f_0^0 = \frac1{r^2}\pd_r(r^2\pd_r f_0^0)$ is the radial part of the Laplace operator in $\R^3$.

In particular, since \eqref{eq-f00} coincides with the system in \cite[(2.6)]{LW-vbaslinear2023} when $\mu_l=0$, 
$\int_{B_1} |f_0^0(|y|,t)|^2dy$, $a_0^0(t)$, and $\dot{a}_0^0(t)$ decay to zero at an exponential rate as $t\to\infty$, by \cite[Theorem 4.1]{LW-vbaslinear2023}.

Moreover, by \cite[Remark 4.2]{LW-vbaslinear2023}, if $f_0^0(|y|,0)\in C^{2+2\al}_y(B_1)$ for some $\al\in(0,1/2)$, we have that $\norm{f_0^0(|y|,t)}_{C^{2+2\al}_y(B_1)}$, $\ddot{a}_0^0(t)$, and $\dddot{a}_0^0(t)$ decay to zero at an exponential rate as $t\to\infty$.
\end{proposition}
\begin{proof}
Using that $Y^0_0=1/(2\sqrt\pi)$, we express \eqref{eq1.3modified-c-lin-irro-0} in the form $2\sqrt{\pi}\mathcal{P}_g Y_0^0 = \Rg T_\infty\varrho$. Then,  projecting this equation onto the radial mode, we obtain
\[2\sqrt{\pi}\mathcal{P}_g(t) = \Rg T_\infty f_0^0(1,t).\]
Plugging this relation into \eqref{eq-a-old-0mode} yields \eqref{eq-f00-c}.

Next, we project \eqref{eq1.2modified-b-lin-irro-0} onto the radial mode to get
\EQN{
\pd_t f_0^0(r,t) &= \frac{\ka}{\ga c_v} \frac1{\rho_*R_*^2} \De_r f_0^0(r,t) + \frac{\rho_*}{\ga p_*} 2\sqrt{\pi}\pd_t\mathcal{P}(t)\\
&= \frac{\ka}{\ga c_v} \frac1{\rho_*R_*^2} \De_r f_0^0(r,t) + \frac1{\ga} \pd_tf_0^0(1,t),\ 0\le r\le1,\ t>0,
}
which is \eqref{eq-f00-a}.

Moreover, multiplying  \eqref{eq1.2modified-a-lin-irro-0} by $Y_0^0$  and integrating the equation over $B_1$, we get
\EQ{\label{eq-dt-linear-mass}
0 &= \int_{B_1} \pd_tf_0^0(|y|,t)\, dy + \frac{\rho_*}{R_*} \int_{B_1} \De\Psi_0^0(|y|,t)\, dy\\
&= \int_{B_1} \pd_tf_0^0(|y|,t)\, dy + \frac{\rho_*}{R_*} \int_{\pd B_1} \pd_r\Psi_0^0\, dS\\
&= \int_{B_1} \pd_tf_0^0(|y|,t)\, dy + 4\pi\, \frac{\rho_*}{R_*} \dot{a}_0^0,
}
where we have used the projection of the equation \eqref{eq1.3modified-a-lin-irro-0} onto the radial mode: $\pd_r\Psi_0^0 = \dot{a}_0^0$ in the last equation.
Integrating \eqref{eq-dt-linear-mass} over time and using \eqref{eq-linear-mass-conserve} yield \eqref{eq-f00-b}.
This completes the proof of the proposition.
\end{proof}

By Proposition \ref{prop-l=0} and \eqref{eq-dot-a}, $b_0^0(t)$ decays exponentially as $t\to\infty$. 
Moreover, $\sup_{r\ge1} \Phi_l(r,\th,\varphi,t) = \Phi_l|_{r=1}$ by the expression \eqref{eq-decompose-Phi-l}.
Thus,
\EQN{
\sup_{r\ge1} \bka{Y_0^0, \Phi_l(r,\cdot,\cdot,t)}_{L^2(\mathbb{S}^2)} 
&= \sup_{r\ge1} \bka{Y_0^0, \Phi_l|_{r=1}}_{L^2(\mathbb{S}^2)}\\
&= b_0^0(t)\ \text{ decays to zero at an exponential rate as $t\to\infty$}.
}
This, together with Proposition \ref{prop-l=0}, yields \eqref{eq-radial-mode}.

\subsubsection{Shape / multipole $(\ell\ge2)$ linearized dynamics inside the gas}\label{sec:shape-dynamics}
\begin{proposition}
Suppose $(\mathcal{R},\Phi_l,\varrho)$ solves \eqref{eq1.1modified-lin-irro-0}--\eqref{eq1.3modified-lin-irro-0} and has the expansions \eqref{eq-decompose-R}--\eqref{eq-Phig-decomposition}.
Then, for $\ell\ge2$, $\varrho_\ell^m$ is a solution to the Dirichlet problem of the heat equation
\EQ{\label{eq-varrho-heat}
\pd_t\varrho_\ell^m &= \frac{\ka}{\ga c_v} \frac1{\rho_*R_*^2} \De\varrho_\ell^m,\quad\text{ in } B_1,\\
\varrho_\ell^m &= 0,\qquad\qquad\qquad\quad\text{ on } \pd B_1.
}
As a consequence, thanks to the classical result for the heat equation, we have \eqref{eq-shape-modes-rho}. 
\end{proposition}
\begin{proof}
Since $\mathcal{P}_g = \mathcal{P}_g(t)$, the projection of equations \eqref{eq1.2modified-b-lin-irro-0} onto $Y^m_\ell$ for $\ell\ge2$ gives the desired equations.
\end{proof}

\subsubsection{Velocity potential of the gas flow}
Projecting the system \eqref{eq-Phi-g-Neumann} onto the modes $\ell=0$ and $\ell\ge2$, and using the fact that the projections commute with $\De$, $\pd_t$, and $\pd_r$, we have that $(\Phi_g)_\ell^m$ satisfies the same Neumann problem:
\EQ{\label{eq-Phi-g-Neumann-modes}
-\De(\Phi_g)_\ell^m &= \frac{R_*}{\rho_*} \pd_t\varrho_\ell^m,\quad\ \text{ in } B_1,\\
\pd_r(\Phi_g)_\ell^m &= \dot{a}_\ell^m(t)Y_\ell^m,\quad\text{ on } \pd B_1.
}
Solving \eqref{eq-Phi-g-Neumann-modes} by using the Green's formula, we have
\[
(\Phi_g)_\ell^m(y,t) = \int_{B_1} N(y,z) \bkt{\frac{R_*}{\rho_*} \pd_t\varrho_\ell^m(z,t)} dz + \int_{\pd B_1} N(y,z) \dot{a}_\ell^m(t)Y_\ell^m(z)\, dS_z,
\]
where $N(y,z)$ is the Neumman-Green function for $-\De$ in $B_1$.

For the monopole / radial mode ($\ell=0$), it follows from \eqref{eq-f00-a} and \eqref{eq-f00-c} that 
\EQN{
\pd_t\varrho_0^0 &=\bkt{\frac{\ka}{\ga c_v}\frac1{\rho_*R_*^2}\De_rf_0^0 + \frac1{\ga}\pd_tf_0^0(1,t)}Y_0^0\\
&=\bkt{\frac{\ka}{\ga c_v}\frac1{\rho_*R_*^2}\De_rf_0^0 + \frac1{\ga\Rg T_\infty} \bke{-\frac{2\si}{R_*^2}\dot{a}_0^0(t) + \rho_lR_*\dddot{a}_0^0(t)}}Y_0^0
}
which decays exponentially to zero uniformly in $B_1$ as $t\to\infty$ by Proposition \ref{prop-l=0}.

For the multipole / shape modes $(\ell\ge2)$, by differentiating \eqref{eq-varrho-heat} with respect to time, it follows that $\pd_t\varrho$ is a solution of the Dirichlet problem of the heat equation.
The standard result for the heat equation (see Lemma \ref{lem-unif-decay-heat}) then implies that, as time advances, $\pd_t\varrho_\ell^m\to0$ uniformly in $B_1$, $\ell\ge2$.

Therefore, we have
\[
\abs{ (\Phi_g)_\ell^m(y,t) - \dot{a}_\ell^m(t) \int_{\pd B_1} N(y,z) Y_\ell^m(z)\, dS_z,
} \lec \int_{B_1} |N(y,z)| | \pd_t\varrho_\ell^m(z,t)| dz \to 0,\ \text{ as } t\to\infty,\]
and 
\[
\abs{ \nb(\Phi_g)_\ell^m(y,t) - \dot{a}_\ell^m(t) \int_{\pd B_1} \nb_y N(y,z) Y_\ell^m(z)\, dS_z,
} \lec \int_{B_1} |\nb_y N(y,z)| | \pd_t\varrho_\ell^m(z,t)| dz \to 0,\ \text{ as } t\to\infty,\]
where we've used the fact that $N(y,\cdot), \nb_yN(y,\cdot)\in L^1(B_1)$ for all $y\in\overline{B_1}$ from the formula \cite[(1.10)]{STT-EMJ2016}
\[
N(y,z) \sim \frac1{|z-y|} + \frac1{\abs{y|z|-\frac{z}{|z|}}} - \ln\abs{1-\bka{y,z}+\Big|y|z|-\frac{z}{|z|}\Big|},\quad \bka{y,z} = y_1z_1 + y_2z_2 + y_3z_3.
\]
These prove \eqref{eq-radial-mode-Phi-g} and \eqref{eq-shape-modes-Phi-g}.

The far-field conditions \eqref{eq-far-field-0} are verified in Appendix \ref{sec-far-field}, completing the proof of Theorem \ref{thm-wellposed}.\qed


\begin{remark}\label{rem:no-flux} We claim that the shape / multipole-mode of the temperature perturbations have zero flux across the boundary.
 Recall the gas temperature $T_g$ is given by $T_g = \frac{p_g}{\Rg} \frac1{\rho_g} = T_\infty + \de\mathcal{T}_g + O(\de^2)$, where 
\[
\mathcal{T}_g = \frac{T_\infty}{p_*} \mathcal{P}_g - \frac{T_\infty}{\rho_*}\varrho.
\]
Consider the spherical harmonic expansion of the perturbation of the temperature
\[
\mathcal{T}_g(r,\th,\varphi,t) = \sum_{\ell=0}^\infty \sum_{|m|\le\ell} (\mathcal{T}_g)_\ell^m(r,\th,\varphi,t),
\]
where $(\mathcal{T}_g)_\ell^m$ is the orthogonal projection of $\mathcal{T}_g$ onto the subspace of $Y_\ell^m$.

For $\ell\ge2$, $(\mathcal{T}_g)_\ell^m= -\frac{T_\infty}{\rho_*}\varrho_\ell^m$. So, by  \eqref{eq-Phi-g-Neumann-modes}
\EQN{
\int_{\pd B_1} \hat{\bf n}\cdot\nb(\mathcal{T}_g)_\ell^m\,dS &= \int_{B_1} \De(\mathcal{T}_g)_\ell^m\, dx
= -\frac{T_\infty}{\rho_*} \int_{B_1} \De \varrho_\ell^m\, dx\\
&= -\frac{T_\infty \ga c_vR_*^2}{\ka} \int_{B_1} \pd_t\varrho_\ell^m\, dx
= \frac{T_\infty \ga c_vR_*\rho_*}{\ka}  \int_{B_1}  \De(\Phi_g)_\ell^m\, dx\\
&= \frac{T_\infty \ga c_vR_*\rho_*}{\ka}  \int_{\pd B_1}  \pd_r(\Phi_g)_\ell^m\, dS
= \frac{T_\infty \ga c_vR_*\rho_*}{\ka}  \dot{a}_\ell^m(t) \int_{\pd B_1}  Y_\ell^m\, dS = 0.
}
\end{remark}

\section{Ill-posedness of the viscous linearized system: Proof of Theorem \ref{thm-linear-illposed}}\label{sec:linear-illposed}

In this section we prove Theorem \ref{thm-linear-illposed}, on the linear ill-posedness for the irrotational and viscous case.
The key observation is that irrotationality and viscosity of the liquid
constrain the liquid velocity potential, $\Phi_l$, to be a radial function.

\begin{proposition}\label{prop-Phil-radial}
If $\mu_l>0$.
Let $(\mathcal{R},\Phi_l,\varrho,\Phi_g)\in C^1([0,\infty);X)$ be a solution of the linearized system \eqref{eq1.1modified-lin-irro-0}--\eqref{eq-far-field-0}. 
Then $\Phi_l = \Phi_l(r,t)$ is independent of $\th$ and $\varphi$.
\end{proposition}
\begin{proof} 
The proof follows from considerations which involve the property that $\Delta_y\Phi_l=0$ for $|y|>1$ and the boundary conditions \eqref{eq1.3modified-b1-lin-irro-1} and  \eqref{eq1.3modified-b2-lin-irro-1}, which hold if $\mu_l>0$.
For readability, and since it plays no role in the argument,  we suppress the time dependence.
By employing \eqref{eq-far-field-0} to drop the boundary term at infinity and using \eqref{eq1.3modified-b2-lin-irro-1} in the last equation and the fact that $\De$ commutes with $\pd_\varphi$, we have
\EQN{
0 &= \int_{\mathbb{R}^3\setminus B_1} (\De\pd_\varphi\Phi_l)(\pd_\varphi\Phi_l)\, dx
= - \int_{\mathbb{R}^3\setminus B_1} |\nb(\pd_\varphi\Phi_l)|^2\, dx
- \int_{\pd B_1} \pd_r(\pd_\varphi\Phi_l)(\pd_\varphi\Phi_l) dS\\
&= - \int_{\mathbb{R}^3\setminus B_1} |\nb(\pd_\varphi\Phi_l)|^2\, dx
- \int_{\pd B_1} (\pd_\varphi\Phi_l)^2 dS,
}
which implies $\pd_\varphi\Phi_l \equiv \text{constant} = \pd_\varphi\Phi_l\big|_{\pd B_1} \equiv 0$.
Hence, $\Phi_l = \Phi_l(r,\th)$. \\
Further,  using  that $\pd_\varphi\Phi_l\equiv 0$, we have
\EQN{
0 &= \int_{\R^3\setminus B_1} (\pd_\th\De\Phi_l)(\pd_\th\Phi_l)\, dx \\
&= \int_0^{2\pi} \int_0^\pi \int_1^\infty \pd_\th\bkt{\frac1{r^2}\pd_r(r^2\pd_r\Phi_l) + \frac1{r^2\sin\th}\pd_\th(\sin\th\pd_\th\Phi_l) + \frac1{r^2\sin^2\th}\pd_\varphi^2\Phi_l} (\pd_\th\Phi_l) r^2 \sin\th\, dr d\th d\varphi\\
&= \int_0^{2\pi} \int_0^\pi \int_1^\infty \pd_r(r^2\pd_\th\pd_r\Phi_l) (\pd_\th\Phi_l)\sin\th\, dr d\th d\varphi \\
&\quad+ \int_0^{2\pi} \int_0^\pi \int_1^\infty \pd_\th\bkt{\frac1{\sin\th}\pd_\th(\sin\th\pd_\th\Phi_l)}(\pd_\th\Phi_l)\sin\th\, dr d\th d\varphi.
}
Integrating by parts, we obtain
\EQN{
0 &= - \int_0^{2\pi} \int_0^\pi \int_1^\infty (\pd_\th\pd_r\Phi_l)^2 r^2\sin\th\, dr d\th d\varphi 
- \int_0^{2\pi} \int_0^\pi \bkt{ (\pd_\th\pd_r\Phi_l)(\pd_\th\Phi_l)} \big|_{r=1} \sin\th\, d\th d\varphi\\
&\quad - \int_0^{2\pi} \int_0^\pi \int_1^\infty \frac1{\sin\th} \bkt{\pd_\th(\sin\th\pd_\th\Phi_l)}^2 dr d\th d\varphi 
+ \int_0^{2\pi} \int_1^\infty \bkt{ \pd_\th(\sin\th\pd_\th\Phi_l)(\pd_\th\Phi_l) }\big|_{\th=0}^{\th=\pi}\, dr d\varphi \\
&= - \int_0^{2\pi} \int_0^\pi \int_1^\infty (\pd_\th\pd_r\Phi_l)^2 r^2\sin\th\, dr d\th d\varphi 
- \int_0^{2\pi} \int_0^\pi \bke{\pd_\th\Phi_l\big|_{r=1}}^2  \sin\th\, d\th d\varphi\\
&\quad - \int_0^{2\pi} \int_0^\pi \int_1^\infty \frac1{\sin\th} \bkt{\pd_\th(\sin\th\pd_\th\Phi_l)}^2 dr d\th d\varphi 
- \int_0^{2\pi} \int_1^\infty \bkt{ \bke{\pd_\th\Phi_l\big|_{\th=\pi}}^2 + \bke{\pd_\th\Phi_l\big|_{\th=0}}^2} dr d\varphi,
}
where we've used \eqref{eq1.3modified-b1-lin-irro-1} and that $\bkt{ \pd_\th(\sin\th\pd_\th\Phi_l)(\pd_\th\Phi_l) }\big|_{\th=0}^{\th=\pi} = \bkt{ \cos\th(\pd_\th\Phi_l)^2 + \sin\th\pd_\th^2\Phi_l\pd_\th\Phi_l }\big|_{\th=0}^{\th=\pi} = -\bke{\pd_\th\Phi_l\big|_{\th=\pi}}^2 - \bke{\pd_\th\Phi_l\big|_{\th=0}}^2$.
Thus, we get $\pd_\th\pd_r\Phi_l \equiv 0 $ and $\pd_\th\Phi_l\big|_{r=1}\equiv 0$.
So, $\pd_\th\Phi_l \equiv \pd_\th\Phi_l\big|_{r=1} \equiv 0$.

Therefore, we have $\Phi_l = \Phi_l(r)$ and complete the proof the proposition.
\end{proof}

We are now ready to prove the linear ill-posedness, Theorem \ref{thm-linear-illposed}.

\begin{proof}[Proof of Theorem \ref{thm-linear-illposed}]
If $\mu_l>0$, then $\Phi_l$ is a radial function by Proposition \ref{prop-Phil-radial}.
Recall $\Phi_l$ has the decomposition \eqref{eq-decompose-Phi-l}.
By projecting $\Phi_l$ onto the shape / multipole modes, we have that the coefficients $b_\ell^m(t) \equiv 0$ for $\ell\ge1$.
Keeping the viscosity $\mu_l$ in \eqref{eq1.3modified-lin-irro-0} and following the same approach deriving \eqref{eq-dot-b-b}, we have
\[
\dot b_\ell^m = -\frac{2\mu_l}{\rho_lR_*^2}(\ell+1)(\ell+2)b_\ell^m + \frac{\si}{\rho_lR_*^3} (\ell+2)(\ell-1) a_\ell^m,\quad \ell\ge1.
\]
That $b_\ell^m(t)\equiv0$ implies that $a_\ell^m(t)\equiv0$ for all $\ell\ge2$, $|m|\le\ell$. 
Further, by  \eqref{eq-linear-centroid}
we have $a_\ell^m(t)\equiv0$ for all $\ell\ge1$, $|m|\le\ell$.
In other words, $\mathcal{R}(\th,\varphi,t) = a_0^0(t) Y_0^0(\th,\varphi) = (2\sqrt\pi)^{-1} a_0^0(t)$ is a radial function.
This yields a contradiction, unless the initial data are radially symmetric. The proof of  Theorem \ref{thm-linear-illposed} is now complete.
\end{proof}

\section{Nonlinear ill-posedness of the viscous irrotational system (\ref{eq1.1simplified-red-irro})--(\ref{eq-far-field-all-irro})}\label{sec-illposed}

Consider the nonlinear, viscous irrotational system \eqref{eq1.1simplified-red-irro}--\eqref{eq-far-field-all-irro}, $\mu_l>0$, with initial bubble surface $\pd\Om(0) = \pd B_{R_*} + \mathcal{R}(\cdot,0)$, where $\mathcal{R}(\cdot,0)$ is small.
Upon a change of variables that first maps the Eulerian coordinates to the Lagrangian coordinates and then maps $\pd B_{R_*} + \mathcal{R}(\cdot,0)$ to $\pd B_1$,
the system \eqref{eq1.1simplified-red-irro}--\eqref{eq-far-field-all-irro} is transformed to a system enclosed in $B_1$.

The transformed system can then be discussed in the function space $X$ defined in \eqref{eq-X}.
Write the transformed system in the following abstract formulation:
\EQ{\label{eq-abstract}
\pd_tu &= \mathcal{N}(u),\quad t>0,\\
u\big|_{t=0}&= u_0.
}
Let $u_*$ be any fixed equilibrium of \eqref{eq-abstract} (see \eqref{eq-equilibrium}), i.e., $\mathcal{N}(u_*)=0$.

\begin{definition}\label{def-well-posed}
We say the system \eqref{eq-abstract} is locally well-posed in $X$ if 
\EN{
\item For any initial data $u_0$, there exists a $T=T(u_0)>0$ such that \eqref{eq-abstract} has a unique solution $u(t)\in X$ for $t\in[0,T)$.
\item The solution $u$ is continuously differentiable with respect to the initial data $u_0$.
}
\end{definition}

\medskip
\noindent{\bf Proof of Theorem \ref{thm-irro-illposed}.}
Suppose \eqref{eq-abstract} is locally well-posed in $X$ in the sense of Definition \ref{def-well-posed}.
Consider the initial data $u_0 = u_*+\ep v_0$, where $v_0$ is some nonspherically symmetric data, and denote the solution by $u(t;\ep)$.
Decompose $u(t;\ep)$ as 
\EQ{\label{eq-u-perturbation}
u(t;\ep) = u_* + \ep v(t;\ep).
}
Since $u(t;\ep)$ is continuously differentiable with respect to the initial data, we have $u(t;\ep)\to u_*$ as $\ep\to0$, which yields $\ep v(t;\ep)\to0$ as $\ep\to0$, and $\lim_{\ep\to0} \frac{d}{d\ep}u(t;\ep) = \lim_{\ep\to0}\bke{ v(t;\ep) + \ep\frac{d}{d\ep}v(t;\ep)}$ exists. 
These implies that, by the L'H\^opital's rule, the limit
\[
\lim_{\ep\to0} v(t;\ep) = \lim_{\ep\to0} \frac{\ep v(t;\ep)}{\ep} = \lim_{\ep\to0} \frac{v(t;\ep) + \ep \frac{d}{d\ep} v(t;\ep)}1 \text{ exists},
\]
which further yields $\lim_{\ep\to0} \ep\frac{d}{d\ep} v(t;\ep)=0$.

Let $T_\ep$ be the existence time of $u(t;\ep)$, i.e., $T_\ep = T(u_*+\ep v_0)$. 
Since that $u(t;\ep)\to u_*$ as $\ep\to0$ and that the existence time of $u(t;0) = u_*$ is $T_0 = \infty$, $T_\ep$ is bounded away from zero for small $\ep$.
Let $T>0$ be a lower bound of $T_\ep$ for $0\le\ep<1$.

Plugging \eqref{eq-u-perturbation} into the equation \eqref{eq-abstract}, differentiating the equation with respect to $\ep$, sending $\ep$ to zero and using all the convergence results obtained above, we obtain 
\EQ{\label{eq-linear-abstract}
\pd_tv(t;0) &= \mathcal{L}v(t;0),\quad t\in[0,T),\\
v(0;0)&= v_0,
}
where $\mathcal{L} = \mathcal{N}'(u_*)$ is the linearized operator.
According to the hypothesis for the contradiction argument, \eqref{eq-linear-abstract} admits a unique solution in $X$.
However, \eqref{eq-linear-abstract} is exactly the same as the linearized system \eqref{eq1.1modified-lin-irro-0}--\eqref{eq1.3modified-lin-irro-0}, which is ill-posed in $X$ by Theorem \ref{thm-linear-illposed}.
This leads to a contradiction, completing the proof of Theorem \ref{thm-irro-illposed}.

\appendix

\section{Frame of reference moving with the bubble centroid}\label{appdx-centroid-frame}
The centroid of volume of the bubble is given by
\[
\boldsymbol{\xi}(t) = \int_{\Om(t)} x\, dx.
\]
We express the dynamics in terms of coordinates moving with the centroid of the bubble, 
given by:
\[
x' = x - \boldsymbol{\xi}(t), 
\]
and hence 
\EQ{\label{eq-centroid=0}
\int_{\Om'(t)} x'\, dx' = {\bf 0},\quad \Om'(t) = \Om(t) - \boldsymbol{\xi}(t).
}
The fluid and gas state variables in this coordinate system are:
\EQN{
{\bf v}_l'(x',t) &= {\bf v}_l(x'+\boldsymbol{\xi}(t),t),\quad\,\, 
p_l'(x',t) = p_l(x'+\boldsymbol{\xi}(t),t),\\
{\bf v}_g'(x',t) &= {\bf v}_g(x'+\boldsymbol{\xi}(t),t),\quad
\rho_g'(x',t) = \rho_g(x'+\boldsymbol{\xi}(t),t),\quad\ \
p_g'(t) = p_g(t),\\
\boldsymbol{\om}'(t) &= \boldsymbol{\om}(t) - \boldsymbol{\xi}(t),\qquad\ \, 
\hat{\bf n}'(x',t) = \hat{\bf n}(x'+\boldsymbol{\xi}(t),t).
}

Then, the equations of reduced irrotational system, \eqref{eq1.1simplified-red-irro}--\eqref{eq1.3simplified-red-irro}, become
 \begin{subequations}\label{eq1.1modified-cm}
\begin{empheq}[right=\empheqrbrace\text{in $\R^3\setminus \Om'(t)$, $t>0$,}]{align}
(\pd_t - \dot{\boldsymbol{\xi}}\cdot\nb_{x'})\phi_l' + \frac{|\nb_{x'}\phi_l'|^2}2 =& - \frac{p_l'-p_{\infty,*}}{\rho_l}, \label{eq1.1modified-a-cm}\\
\De_{x'}\phi_l' =&\, 0, \label{eq1.1modified-b-cm}
\end{empheq}
\end{subequations}

{\small
\begin{subequations}\label{eq1.2modified-cm}
\begin{empheq}[right=\empheqrbrace\text{in $\Om'(t)$, $t>0$,}]{align}
&(\pd_t - \dot{\boldsymbol{\xi}}\cdot\nb_{x'})\rho_g' + \rho_g'\De_{x'}\phi_g' + \nb_{x'}\rho_g'\cdot\nb_{x'}\phi_g' = 0,\label{eq1.2modified-a-cm}\\
&(\pd_t - \dot{\boldsymbol{\xi}}\cdot\nb_{x'}) \rho_g' = \frac{\ka}{\ga c_v} \De_{x'}\log\rho_g' - \frac{\ka}{\ga c_v} \frac{|\nb_{x'}\rho_g'|^2}{(\rho_g')^2} - \nb_{x'}\phi_g'\cdot\nb_{x'}\rho_g' + \frac{\pd_tp_g'}{\ga p_g'}\rho_g', \label{eq1.2modified-b-cm}
\end{empheq}
\end{subequations}
}
and
 \begin{subequations}\label{eq1.3modified-cm}
\begin{empheq}[right=\empheqrbrace\text{on $\pd\Om'(t)$, $t>0$.}]{align}
\nb_{x'}\phi_l'(\boldsymbol{\om}',t)\cdot\hat{\bf n}' = \nb_{x'}\phi_g'(\boldsymbol{\om}',t)\cdot\hat{\bf n}' = \pd_t \boldsymbol{\om}' \cdot\hat{\bf n}' + \dot{\boldsymbol{\xi}}\cdot\hat{\bf n}', \label{eq1.3modified-a-cm}\\
p_g' - p_l' = \si (\nb_S\cdot \hat {\bf n}'), \label{eq1.3modified-b-cm}\\
p_g' = \Rg T_\infty\rho_g', \label{eq1.3modified-c-cm}
\end{empheq}
\end{subequations}

In addition to linearizing the system for the state variables via \eqref{eq-linearize-old}, we expand 
\[\boldsymbol{\xi} = \de\boldsymbol{\Xi} + O(\de^2).\]
Then we derive the linearized system 
\begin{subequations}\label{eq1.1modified-lin-irro-1-cm}
\begin{empheq}[right=\empheqrbrace\text{in $\R^3\setminus B_1$, $t>0$,}]{align}
\pd_t \Phi_l =& - \dfrac1{\rho_*R_*}\, \mathcal P_l, \label{eq1.1modified-a-lin-irro-1-cm}\\
\De\Phi_l =&\, 0,\quad 
\label{eq1.1modified-b-lin-irro-1-cm}
\end{empheq}
\end{subequations}

 \begin{subequations}\label{eq1.2modified-lin-irro-1-cm}
\begin{empheq}[right=\empheqrbrace\text{in $B_1$, $t>0$,}]{align}
\pd_t \varrho + \frac{\rho_*}{R_*} \De\Phi_g =&\, 0,\qquad \mathcal{P}_g = \mathcal{P}_g(t),\label{eq1.2modified-a-lin-irro-1-cm}\\
\pd_t \varrho =&\, \dfrac{\ka}{\ga c_v} \dfrac1{\rho_*R_*^2} \De \varrho + \dfrac{\rho_*}{\ga p_*}\pd_t\mathcal P_g, \label{eq1.2modified-b-lin-irro-1-cm}
\end{empheq}
\end{subequations}
and
 \begin{subequations}\label{eq1.3modified-lin-irro-1-cm}
\begin{empheq}[right=\empheqrbrace\text{on $\pd B_1$, $t>0$,}]{align}
\pd_r\Phi_l =&\, \pd_r\Phi_g = \pd_t\mathcal R + \dot{\boldsymbol{\Xi}}\cdot\hat{\bf r}, \label{eq1.3modified-a-lin-irro-1-cm}\\
\mathcal P_g - \mathcal P_l =& -\frac{\si}{R_*^2}(2+\De_S)\mathcal R, \label{eq1.3modified-b-lin-irro-1-cm}\\
\mathcal P_g =&\, \Rg T_\infty \varrho, \label{eq1.3modified-c-lin-irro-1-cm}
\end{empheq}
\end{subequations}

At order $\delta$, the zero-centroid-of-volume condition, \eqref{eq-centroid=0}, implies the following three orthogonality constraints  on  $\mathcal{R}$
 (see e.g. \cite[(C.28)--(C.30)]{SW-SIMA2011}):
\EQ{\label{eq-linear-centroid=0}
\bka{\mathcal R, Y_1^m}_{L^2(\mathbb S^2)} = 0,\qquad m=-1,0,1,
}
where $\bka{f,g}_{L^2(\mathbb S^2)} =  \int_0^{2\pi} \int_0^\pi f\bar g\sin\th\, d\th d\varphi$.

\begin{proposition}\label{prop-Xi=0}
Let $(\Phi_l, \mathcal P_l, \Phi_g, \mathcal P_g, \mathcal R, \varrho, {\bf \Xi})$ denote a sufficiently regular solution of the initial-boundary value problem for the system \eqref{eq1.1modified-lin-irro-1-cm}--\eqref{eq1.3modified-lin-irro-1-cm} with the initial data $\Phi_l\big|_{t=0}=0$ and $\boldsymbol{\Xi}\big|_{t=0} = {\bf 0}$, the far-field conditions \eqref{eq-far-field-0}, and the linearized zero-centroid-of-volume condition \eqref{eq-linear-centroid=0}.
Then, 
\[
{\bf \Xi}(t) \equiv {\bf 0},\qquad t\ge0,\qquad t\ge0.
\]
\end{proposition}
\begin{proof}
Define the projection operators onto the $|m|=1$ spherical harmonics,
\[
\mathbb P_1^m = \bka{\cdot,Y_1^m}_{L^2(\mathbb S^2)} Y_1^m.
\]
and denote
\[
V_l = \mathbb P_1^m\pd_t\Phi_l,\qquad
V_g = \mathbb P_1^m\pd_t\Phi_g,\qquad
U = \mathbb P_1^m\pd_t\varrho.
\]

First, applying $\mathbb P_1^m$ to \eqref{eq1.1modified-a-lin-irro-1-cm} yields
\EQ{\label{eq-Vl-proj}
V_l = -\frac1{\rho_lR_*} \mathbb P_1^m \mathcal P_l.
}
The far-field condition \eqref{eq-far-field-0} for $\mathcal{P}_l$ then gives $V_l\to0$ as $|y|\to\infty$.
Next, applying $\mathbb P_1^m\pd_t$ to \eqref{eq1.1modified-b-lin-irro-1-cm} implies that $V_l$ is harmonic in $\mathbb{R}^3\setminus B_1$ since $\mathbb P_1^m$ commutes with $\pd_t$ and $\De$, for any $|m|\le1$. 
Moreover, by applying $\mathbb P_1^m$ to \eqref{eq1.3modified-b-lin-irro-1-cm}, we derive
\EQ{\label{eq-Pl-proj}
- \mathbb P_1^m\mathcal{P}_l\big|_{\pd B_1} = \mathbb P_1\bkt{\frac{\si}{\rho_*R_*^2}(2+\De_S)\mathcal{R}} = 0,
}
where we've used the facts that $\mathcal{P}_g = \mathcal{P}_g(t)$ so that $\mathbb P_1^m\mathcal{P}_g\big|_{\pd B_1} = 0$ and that $(2+\De_S)Y_\ell^m = -(\ell+2)(\ell-1) Y_\ell^m$ so that $(2+\De_S)Y_1^m =0$.
Evaluating \eqref{eq-Vl-proj} on the boundary $\pd B_1$ and using \eqref{eq-Pl-proj} yields $V_l\big|_{\pd B_1} = 0$.
Then $V_l\equiv 0$ in $\R^3\setminus B_1$, for all $t\ge0$, by the maximum principle for harmonic functions.
In particular, $\mathbb P_1^m\Phi_l = \mathbb P_1^m\Phi_l\big|_{t=0}\equiv0$ for all $t\ge0$.

Finally, applying $\mathbb P_1^m$ to \eqref{eq1.3modified-a-lin-irro-1-cm} and using the fact that $\mathbb P_1^m$ commutes with $\pd_r$, we deduce
\EQ{\label{eq-P1Xi}
\mathbb P_1^m [\dot{\bf\Xi}(t)\cdot\hat{\bf r} ] = 0,\qquad |m|\le 1,
}
where we've used the linearized zero-centroid-of-volume condition \eqref{eq-linear-centroid=0} so that $\mathbb P_1^m\mathcal R = 0$.
Since 
\[
\hat{\bf r} = 
\begin{pmatrix}
\sin\th\cos\varphi\\ 
\sin\th\sin\varphi\\ 
\cos\th
\end{pmatrix}
= \sqrt{\frac{2\pi}3} 
\begin{pmatrix}
Y_1^{-1}-Y_1^1\\ 
i(Y_1^{-1} + Y_1^1)\\ 
\frac1{\sqrt2} Y_1^0
\end{pmatrix}
,\] 
\eqref{eq-P1Xi} implies that for $|m|\le1$
\[
0 = \bka{ \dot{\bf\Xi}(t)\cdot\hat{\bf r} , Y_1^m}_{L^2(\mathbb S^2)}
= \sqrt{\frac{2\pi}3}\, \left< \bkt{\dot\Xi_1(Y_1^{-1} - Y_1^1) + i\dot\Xi_2(Y_1^{-1} + Y_1^1) + \frac1{\sqrt2}\, \dot\Xi_3 Y_1^0} , Y_1^m\right>_{L^2(\mathbb S^2)},
\]
where ${\bf \Xi}(t) = (\Xi_1(t), \Xi_2(t), \Xi_3(t))$. This yields that $\dot\Xi_1(t) = \dot\Xi_2(t) = \dot\Xi_3(t) = 0$ for all $t\ge0$. 
By the choice of the initial data ${\bf \Xi}\big|_{t=0} = {\bf 0}$, we conclude that ${\bf \Xi}(t) = {\bf 0}$ for all $t\ge0$, completing the proof of the proposition.

\end{proof}

\section{Uniform decay of heat solutions in a bounded domain}

\begin{lemma}\label{lem-unif-decay-heat}
Consider the Dirichlet problem of heat equation in a bounded domain $\Om$:
\EQ{\label{eq-heat}
\pd_t u &= \bar\ka \De u,\ \text{ in }\Om,\\
u& = 0,\quad\ \ \text{ on }\pd\Om.
}
Then $\max_{x\in\overline{\Om}}|u(x,t)|\to 0$ as $t\to\infty$.
\end{lemma}
\begin{proof}
First of all, by the parabolic smoothing for the heat equation, we may assume the smoothness of $u$. 
Next, using the energy method and Poincar\'e inequality, $\norm{u(t)}_2\to0$ as $t\to\infty$.

Let $M(t) = \max_{x\in\overline{\Om}} |u(x,t)| = u(y_t,t)$ for some $y_t\in\overline{\Om}$.
By the parabolic maximum principle and the Dirichlet boundary condition, the function $M(t)$ is nonincreasing in $t$.
Since $M(t)$ is bounded below by $0$, $\lim_{t\to\infty}M(t) = C\ge0$ exists. 

We claim that $C=0$. Indeed, if $C>0$, there exists $T>0$ such that $M(t)>C/2$ for all $t\ge T$. 
Since $\overline{\Om}$ is compact and $y_t\in\overline{\Om}$, there exists a sequence $T\le t_1<t_2<\ldots$, such that $x_k:= y_{t_k}\to x_\infty\in\overline{\Om}$ as $k\to\infty$.
Note that $|u(x_k,t_k)| = M(t_k)>C/2$ since $t_k\ge T$.
Moreover, by the continuity of $u$, there exists $\de>0$ such that $|u(x,t) - u(x_\infty,t)|<C/8$ for all $x\in B_\de(x_\infty)\cap\overline{\Om}$.
For $k$ large enough such that $x_k\in B_\de(x_\infty)\cap\overline{\Om}$, we have $|u(x_k,t) - u(x_\infty,t)|<C/8$, which implies that $|u(x,t) - u(x_k,t)|< C/4$ for all $x\in B_\de(x_\infty)\cap\overline{\Om}$.
This leads to a contradiction since
\EQN{
\norm{u(t_k)}_2^2 = \int_{\Om} |u(x,t_k)|^2\, dx &\ge \int_{B_\de(x_\infty)\cap\overline{\Om}} |u(x,t_k)|^2\, dx\\
&\ge \int_{B_\de(x_\infty)\cap\overline{\Om}} \bke{|u(x_k,t_k)| - \frac{C}4}^2\, dx
> \frac{C^2}{16} \abs{B_\de(x_\infty)\cap\overline{\Om}},
}
where the left hand side vanishes as $t\to\infty$.
Therefore, $C=0$, completing the proof of Lemma \ref{lem-unif-decay-heat}.
\end{proof}

\begin{remark}[Exponential decay rate]
If the initial data has certain regularity, the decay of the solution $u$ to the Dirichlet problem of the heat equation \eqref{eq-heat} can be shown exponential.
Indeed, expanding $u$ with respect to Dirichlet eigenfunctions $\phi_j$ of $-\De$, $-\De\phi_j = \la_j\phi_j$ on $\Om$, $\phi_j\big|_{\pd\Om} = 0$, with $0<\la_1\le\la_2\le\ldots$ and  $\norm{\phi_j}_{L^2(\Om)}=1$, we obtain $u(x,t) = \sum_{j=1}^\infty c_je^{-\la_j t}\phi_j(x)$, where $c_j = \int_\Om u(x,0)\phi_j(x)\, dx$.
By the Sobolev inequality, since $4>3/2$,
\EQN{
\norm{u(t)}_{L^\infty(\Om)} &\lec \norm{u(t)}_{W^{4,2}(\Om)} = \norm{(I-\De)^2 u(t)}_{L^2(\Om)}
= \norm{\sum_{j=1}^\infty c_j e^{-\la_j t} (I-\De)^2\phi_j}_{L^2(\Om)}\\
&\le \sum_{j=1}^\infty |c_j| e^{-\la_j t} \norm{ (I-\De)^2\phi_j}_{L^2(\Om)}
= \sum_{j=1}^\infty |c_j| e^{-\la_j t} \norm{ (1+\la_j)^2\phi_j}_{L^2(\Om)}\\
&= \sum_{j=1}^\infty |c_j| e^{-\la_j t} (1+\la_j)^2 
\le e^{-\la_1 t} \sum_{j=1}^\infty |c_j| (1+\la_j)^2,
}
which decays exponentially as $t\to\infty$ provided the data satisfies $\sum_{j=1}^\infty |c_j| (1+\la_j)^2<\infty$.
\end{remark}

\section{Verification of far-field conditions}\label{sec-far-field}

In this appendix, we verify the far-field conditions \eqref{eq-far-field-0} for the solution constructed in Section \ref{sec-linear-wellposed}, in the proof of Theorem \ref{thm-wellposed}.

Note that it follows from \eqref{eq-dot-a} and \eqref{eq-a-soln} that
\[
b_\ell^m = -\frac1{\ell+1}\, \dot a_\ell^m \sim \frac{\sqrt{(\ell+2)(\ell+1)(\ell-1)}}{\ell+1} \sim \sqrt{\ell}.
\]
Since $|\pd_\th^a\pd_\varphi^b Y_\ell^m(\th,\varphi)| = O(\ell^{\frac12+a+b})$ for $a+b=0,1,2$ (see e.g. \cite[Proposition 4.11]{SW-SIMA2011}),
\EQN{
\nb\Phi_l &= \pd_r\Phi_l \hat{\bf r} + \frac1r \pd_\th\Phi_l \hat{\boldsymbol{\th}} + \frac1{r\sin\th} \pd_\varphi\Phi_l\hat{\boldsymbol{\varphi}}\\
&= \sum_{\ell=0}^\infty \sum_{|m|\le\ell} b_\ell^m \bke{- \frac{\ell+1}{r^{\ell+2}} Y_\ell^m \hat{\bf r} + \frac1{r^{\ell+2}} \pd_\th Y_\ell^m + \frac1{r^{\ell+2}\sin\th} \pd_\varphi Y_\ell^m}
}
implies that, for $r\ge1$ and fixed $\th\in(0,\pi)$,
\EQN{
|\nb\Phi_l| 
&\le \frac1{r^2} \sum_{\ell=0}^\infty \sum_{|m|\le\ell} \frac{|b_\ell^m|}{r^\ell}\bke{(\ell+1)|Y_\ell^m| + |\pd_\th Y_\ell^m| + \frac{|\pd_\varphi Y_\ell^m|}{|\sin\th|}} \\
&\lec_{\hspace{-0.1cm}\th}\, \frac1{r^2} \sum_{\ell=0}^\infty \sum_{|m|\le\ell} \frac{\sqrt{\ell}\cdot\ell^{\frac32}}{r^\ell}
= \frac1{r^2} \sum_{\ell=0}^\infty  \frac{\ell^3}{r^\ell}
= \frac1{r^2} \cdot \frac{r^4+4r^3+1}{(r-1)^4} = O(r^{-2})\ \text{ as }r\to\infty.
}
If $\th = 0$ or $\pi$, then $y_1=y_2=0$, so $\hat{\bf r} = (y_3/|y_3|) [0,0,1]^\top$, $\hat{\boldsymbol{\th}} = \hat{\boldsymbol{\varphi}} = {\bf 0}$.
Thus, $|\nb\Phi_l|=|\pd_r\Phi_l|\lec O(r^{-2})$ by following the same argument above.
For $D^2\Phi_l$, note that
\EQN{
D^2\Phi_l 
&= \pd_r^2\Phi_l\hat{\bf r}\otimes\hat{\bf r} + \bke{\frac1r\pd_\th\pd_r\Phi_l - \frac1{r^2}\pd_\th\Phi_l}\hat{\bf r}\otimes\hat{\bf\boldsymbol{\th}} + \bke{\frac1{r\sin\th}\pd_\varphi\pd_r\Phi_l - \frac1{r^2\sin\th}\pd_\varphi\Phi_l}\hat{\bf r}\otimes\hat{\boldsymbol{\varphi}}\\
&\quad + \pd_r\bke{\frac1r\pd_\th\Phi_l}\hat{\boldsymbol{\th}}\otimes\hat{\bf r} + \bke{\frac1{r^2}\pd_\th^2\Phi_l + \frac1r\pd_r\Phi_l}\hat{\boldsymbol{\th}}\otimes\hat{\boldsymbol{\th}} + \bke{\frac1{r^2\sin\th}\pd_\varphi\pd_\th\Phi_l - \frac{\cot\th}{r^2\sin\th}\pd_\varphi\Phi_l}\hat{\boldsymbol{\th}}\otimes\hat{\boldsymbol{\varphi}}\\
&\quad + \pd_r\bke{\frac1{r\sin\th}\pd_\varphi\Phi_l}\hat{\boldsymbol{\varphi}}\otimes\hat{\bf r} + \frac1r\pd_\th\bke{\frac1{r\sin\th}\pd_\varphi\Phi_l}\hat{\boldsymbol{\varphi}}\otimes\hat{\boldsymbol{\th}} \\
&\quad + \bke{\frac1{r^2\sin^2\th}\pd_\varphi^2\Phi_l + \frac{\cot\th}{r^2}\pd_\th\Phi_l + \frac1r\pd_r\Phi_l}\hat{\boldsymbol{\varphi}}\otimes\hat{\boldsymbol{\varphi}},
}
which yields the decay $D^2\Phi_l(y,t)=O(|y|^{-3})$ as $|y|\to\infty$ by a similar argument above.
As for $\mathcal{P}$, by using \eqref{eq-dot-a} and \eqref{eq-a-old-hmode}, we have
\EQN{
\dot b_\ell^m = - \frac1{\ell+1}\, \ddot a_\ell^m
 = \frac1{\ell+1}\, \frac{\si}{\rho_lR_*^2}\, (\ell+2)(\ell+1)(\ell-1) a_\ell^m
 = \frac{\si}{\rho_lR_*^2}\, (\ell+2)(\ell-1) \sim \ell^2.
}
Therefore, it follows from \eqref{eq-Pl-expand} that
\EQN{
|\mathcal P_l| &\le \rho_lR_* \sum_{\ell=0}^\infty \sum_{|m|\le\ell} |\dot b_\ell^m| \bke{\frac1r}^{\ell+1} |Y_\ell^m|\\
&\lec \frac1r \sum_{\ell=0}^\infty \sum_{|m|\le\ell} \frac{\ell^2\cdot\ell^{\frac12}}{r^\ell}
= \frac1r \sum_{\ell=0}^\infty \frac{\ell^{\frac72}}{r^\ell}
\le \frac1r \sum_{\ell=0}^\infty \frac{\ell^4}{r^\ell}
= \frac1r \cdot \frac{2(r^3 + 7r^4 + 4r^5)}{(r-1)^5}
= O(r^{-1})\ \text{ as }r\to\infty.
}

\bigskip

\end{document}